\documentclass[12pt,a4paper]{amsart}
\usepackage{amsmath,amssymb,amsthm,latexsym}
\usepackage[abbrev,backrefs]{amsrefs}
\usepackage{a4wide}
\usepackage{mathrsfs}
\usepackage{color}
\usepackage{graphicx}
\usepackage{pifont}
\usepackage[normalem]{ulem}
\usepackage{color}
\usepackage[margin=0.8in]{geometry}
\usepackage{tikz}
\usetikzlibrary{patterns.meta}

\newtheorem{theorem}{Theorem}[section]
\newtheorem{corollary}{Corollary}[section]
\newtheorem{lemma}{Lemma}[section]
\newtheorem{proposition}{Proposition}[section]
\theoremstyle{definition}

\theoremstyle{remark}
\newtheorem{remark}{Remark}[section]

\numberwithin{equation}{section}

\newcommand{\bbr}{\mathbb{R}}
\newcommand{\bbrn}{\mathbb{R}^N}
\newcommand{\bbn}{\mathbb{N}}
\newcommand{\ve}{\varepsilon}

\newcommand{\N}{{\mathbb N}}

\newcommand{\R}{{\mathbb R}}

\newcommand{\D}{\mathcal{D}}
\newcommand{\eps}{{\varepsilon}}
\newcommand{\C}{{\mathscr{C}}}

\renewcommand{\O}{\mathcal{O}}

\begin{document}
\title[Ground states of a nonlocal variational problem]{Ground states of a nonlocal variational problem \\and Thomas--Fermi limit for the Choquard equation}

\author{Damiano Greco}
\address{Department of Mathematics\\ Swansea University\\ Fabian Way\\
	Swansea SA1~8EN\\ Wales, United Kingdom}	
\email{2037214@Swansea.ac.uk}

\author{Yanghong Huang}
\address{Department of Mathematics\\ University of Manchester\\ Oxford Rd\\ Manchester M13 9PL\\ United Kingdom}
\email{yanghong.huang@manchester.ac.uk} 

\author{Zeng Liu}
\address{Department of Mathematics\\ Suzhou University of Science and Technology\\ Suzhou 215009\\ P.R. China}
\email{zliu@mail.usts.edu.cn}

\author{Vitaly Moroz}
\address{Department of Mathematics\\ Swansea University\\ Fabian Way\\
	Swansea SA1~8EN\\ Wales, United Kingdom}	
\email{v.moroz@swansea.ac.uk}

\date{\today}



\begin{abstract}
We study nonnegative optimizers of a Gagliardo--Nirenberg type inequality 
$$\iint_{\mathbb{R}^N \times \mathbb{R}^N} \frac{|u(x)|^p\,|u(y)|^p}{|x - y|^{N-\alpha}} dx\, dy\le C\Big(\int_{{\mathbb R}^N}|u|^2 dx\Big)^{p\theta} \Big(\int_{{\mathbb R}^N}|u|^q dx\Big)^{2p(1-\theta)/q},$$ 
that involves the nonlocal Riesz energy with $0<\alpha<N$, $p>\frac{N+\alpha}{N}$, $q>\frac{2Np}{N+\alpha}$ and $\theta=\frac{(N+\alpha)q-2Np}{Np(q-2)}$. For $p=2$, the equivalent problem has been studied in connection with
the Keller--Segel diffusion--aggregation models in the past few decades. The general case $p\neq 2$ considered here 
appears in the study of Thomas--Fermi limit regime for the Choquard equations with local
repulsion. We establish optimal ranges of parameters for the validity of the above interpolation inequality, discuss the existence and qualitative properties of the
nonnegative maximizers, and in some special cases estimate the optimal constant. For $p=2$ it is known that the maximizers are H\"older continuous and compactly supported on a ball. We
show that for $p<2$ the maximizers are smooth functions supported on $\R^N$,  while for $p>2$ the maximizers  consist of a
characteristic function of a ball and a nonconstant nonincreasing H\"older continuous function supported on the same ball. We use these qualitative properties
of the maximizers to establish the validity of the Thomas--Fermi approximations for the Choquard equations with local repulsion. The results are verified numerically with extensive examples.
\end{abstract}
\maketitle

\section{Introduction}

\subsection{Background}
Our starting point is the Choquard type equation
\begin{equation}\tag{$P_\eps$}\label{eqPeps}
	-\Delta w +\eps w+|w|^{q-2}w=(I_{\alpha}*|w|^p)|w|^{p-2}w\quad\text{in $\R^N$},
\end{equation}
where $w:\R^N\to\R$ is an unknown function, $N\geq 3$, $p>1$, $q>2$ and $\eps\ge 0$. By $I_\alpha(x):=A_\alpha|x|^{\alpha-N}$ we denote the Riesz potential with $\alpha\in(0,N)$, and $*$ stands for the standard convolution in $\R^N$. The choice of the {\em Riesz constant} $A_\alpha:=\frac{\Gamma((N-\alpha)/2)}{\pi^{N/2}2^{\alpha}\Gamma(\alpha/2)}$
ensures that $I_\alpha(x)$ could be interpreted as the Green's function of the fractional Laplacian operator $(-\Delta)^{\alpha/2}$ in $\R^N$, and that the semigroup property $I_{\alpha+\beta}=I_\alpha*I_\beta$ holds for all $\alpha,\beta\in(0,N)$ such that $\alpha+\beta<N$, see for example \cite{DuPlessis}*{pp.\thinspace{}73--74}.

When $N=3$, $\alpha=2$, $p=2$ and $q=4$, under the name of Gross--Pitaevskii--Poisson  equation, \eqref{eqPeps} was proposed in cosmology as a model to describe the Cold Dark Matter made of axions or bosons in the form of self--gravitating  Bose--Einstein Condensate  at absolute zero temperatures \cite{Wang,Bohmer-Harko,Chavanis-11,Braaten,Eby}.
The nonlocal convolution term on the right hand side of  \eqref{eqPeps} represents the gravitational attraction between bosonic particles. The
local term $|w|^{q-2}w$ accounts for the repulsive short--range quantum force self--interaction between bosons. Similar models appear in the literature under the names
Ultralight Axion Dark Matter, and Fuzzy Dark Matter, see \cite{Braaten} for a history survey. More generally, \eqref{eqPeps} can be seen as a Hartree type
nonlinear Schr\"{o}dinger equation with an attractive long-range interaction and repulsive short-range interactions. 
While for most of the relevant physical applications the parameter $p$ is chosen to be $2$, the cases with $p\neq 2$ appear in several relativistic models of the density functional theory \citelist{\cite{BPO-2002}\cite{BLS-2008}\cite{BGT-2012}}.

By a {\em ground state} of \eqref{eqPeps} we understand a nonnegative weak solution $w \in H^1(\R^N)\cap L^q(\R^N)$ which has a minimal energy of the functional
\begin{equation}\label{eq02}
	\mathcal{I}_{\eps}(w):=\frac{1}{2}\int_{\R^N}|\nabla w|^2dx+\frac{\eps}{2}\int_{\R^N}|w|^2dx+\frac{1}{q}\int_{\R^N}|w|^q dx-\frac{1}{2p}\int_{\R^N}(I_{\alpha}*|w|^p)|w|^p dx
\end{equation}
amongst all nontrivial finite energy solutions of \eqref{eqPeps}.
The following was proved in \cite{ZLVMeps}*{Theorem 1.1}, under optimal or near optimal assumptions on the parameters.

\begin{theorem}\label{thm01eps}
	Let $\frac{N+\alpha}{N}<p<\frac{N+\alpha}{N-2}$ and $q>2$, or $p\geq\frac{N+\alpha}{N-2}$ and $q>\frac{2Np}{N+\alpha}$.
  Then for each $\eps>0$,  \eqref{eqPeps} admits a positive, radial, monotone decreasing ground state solution $w_\eps \in H^1(\R^N)\cap L^1(\R^N)\cap C^2(\R^N)$.
	Moreover, there exists a positive constant $C$ depending on $N$, $\alpha$, $p$, $q$ and $\eps$, such that
	\begin{itemize}
		\item if \(p > 2\),
		\[
		\lim_{|x| \to \infty} w_\eps(x) |x|^{\frac{N - 1}{2}} e^{\sqrt{\eps}|x|}=C;
		\]
		\item if \(p = 2\),
		\[
      \lim_{|x| \to \infty} w_\eps(x) |x|^{\frac{N - 1}{2}} \exp \left(\int_{\rho_\eps}^{|x|} \sqrt{\eps - \tfrac{A_\alpha\|w_\eps\|_2^2}{s^{N - \alpha}}}
      \,ds\right)=C,
      \qquad \rho_\eps=\Big(\eps^{-1}A_\alpha\|w_\eps\|_2^2\Big)^{\frac{1}{N - \alpha}};
		\]
		\item if \(p < 2\),
		$$\lim_{x\to\infty}w_\eps(x)|x|^\frac{N-\alpha}{2-p}=\Big(\eps^{-1}A_\alpha\|w_\eps\|_p^p\Big)^{\frac1{2-p}}.$$
	\end{itemize}
\end{theorem}

In addition to the existence of ground states for every fixed $\eps>0$, in \cite{ZLVMeps} the authors have identified and studied several limit regimes for
ground states of \eqref{eqPeps}, as  $\eps\to 0$ or $\eps\to\infty$. One of  the relevant limit regimes is associated with the rescaling
\begin{equation}\label{TF-scale}
u_\eps(x):=\eps^{-\frac{1}{q-2}}w\Big(\eps^{-\frac{2p-q}{\alpha(q-2)}}x\Big),
\end{equation}
that converts \eqref{eqPeps} to the equation
\begin{equation}\label{eqPeps-TF}
	-\eps^{\nu}\Delta u_\eps+u_\eps+|u_\eps|^{q-2}u_\eps=(I_{\alpha}*|u_\eps|^p)|u_\eps|^{p-2}u_\eps\quad\text{ in $\R^N$},
\end{equation}
where $\nu=\frac{2(2p+\alpha)-q(2+\alpha)}{\alpha(q-2)}$. The {\em Thomas--Fermi limit regime} for the Choquard equation \eqref{eqPeps} is the scenario when
$\eps\to 0$ and $\nu>0$, or $\eps\to\infty$ and $\nu<0$. In this regime, $\eps^{\nu}$ approaches zero and 
the {\em formal} limit equation for Eq. \eqref{eqPeps-TF} is the {\em Thomas--Fermi} type integral equation
\begin{equation*}\tag{$T\!F$}\label{eqTF}
u+|u|^{q-2}u=(I_{\alpha}*|u|^p)|u|^{p-2}u\quad\text{ in $\R^N$}.
\end{equation*}
When $p=2$ and $\alpha=2$, equations equivalent to  \eqref{eqTF} are well-known in astrophysical literature, cf. \cite[p.92]{Chandrasekhar},
and their mathematical analysis  goes back to \citelist{\cite{Auchmuty-Beals}\cite{Lions-81}}. More recently \eqref{eqTF} with $p=2$ and
general $\alpha\in(0,N)$ was studied in \citelist{\cite{Carrillo-CalcVar}\cite{Carrillo-NA}} (existence of solutions) and
\citelist{\cite{Carrillo-Sugiyama}\cite{Volzone}\cite{Carrillo-NA-2021}}  (uniqueness), all in the context of Keller--Segel models. See also \cite[Theorem 2.6]{ZLVMeps} which proves the existence of a ground state for \eqref{eqTF} with $p=2$ for the optimal range $q>\frac{4N}{N+\alpha}$, extending some of the existence results in \citelist{\cite{Carrillo-CalcVar}\cite{Carrillo-NA}}. 

In \cite[Theorems 2.7 and 3.2]{ZLVMeps}, for the special case $p=2$ and $\alpha=2$ the authors established the convergence of the rescaled
ground states  $u_\eps$ of \eqref{eqPeps-TF} to the ground states $u$ of \eqref{eqTF} in the Thomas--Fermi regime, thus justifying the formal analysis of the rescaling \eqref{TF-scale}. Recall that for $p=2$ the limit ground state of \eqref{eqTF} is compactly supported on a ball, so that the rescaled ground states $u_\eps$ develop a sharp ``corner'' near the boundary of the support of the limit ground state. This phenomenon is well-known in astrophysics, where the radius of support of the limit ground state provides approximate radius of the astrophysical object. In the context of self--gravitating Bose--Einstein Condensate models, the Thomas--Fermi limit regime (under the name of {\em Thomas--Fermi approximations}) was used as the key tool in the astrophysical studies of the Gross--Pitaevskii--Poisson equation ($\alpha=2$, $p=2$, $q=4$) in \cite{Wang,Bohmer-Harko,Chavanis-11}.
\smallskip

The main goal of this work is to study the existence and qualitative properties of ground states for  \eqref{eqTF} for $p\neq 2$ and $\alpha\in(0,N)$, under optimal assumptions on the parameters, and to use these properties to establish the validity of the Thomas--Fermi approximations for the Choquard equations with local repulsion. In particular, we are going to prove that:
\begin{itemize}
  \item if $p<2$, ground states for  \eqref{eqTF} are positive smooth functions supported on $\R^N$;\smallskip
  \item if $p>2$, ground states for  \eqref{eqTF} are discontinuous and represented as a linear combination of a characteristic function of a ball, and a non-constant nonincreasing H\"older continuous function supported on the same ball. 
\end{itemize}
As a comparison, for the special case $p=2$, it is well--known in \citelist{\cite{Carrillo-CalcVar}\cite{Carrillo-NA}} that ground states for  \eqref{eqTF} are H\"older continuous and compactly supported on a ball.  
We also establish qualitative properties of the ground states, including decay at infinity for $p<2$, and regularity near the boundary of the support for
$p > 2$. This information becomes crucial in the proofs of convergence of the rescaled ground states of  \eqref{eqPeps} to the limit profiles governed by  \eqref{eqTF}.

\subsection{Variational setup for \eqref{eqTF} and main results}
Solutions of the Thomas--Fermi equation \eqref{eqTF} correspond, at least formally, to the critical points of the energy
$$
E(u)=\frac{1}{2}\int_{\bbr^N}|u|^2dx+\frac{1}{q}\int_{\bbr^N}|u|^q\,dx-\frac{1}{2p}\mathcal{D}_\alpha(|u|^p,|u|^p),
$$
where, and in what follows, $\D_\alpha$ denotes the Coulomb interaction
$$\mathcal{D}_\alpha(f,g):=A_\alpha\iint_{\R^N\times\R^N}\frac{f(x)\,g(y)}{|x-y|^{N-\alpha}}dx\,dy,$$
here $A_\alpha$ is the Riesz constant. Throughout this work, we assume the following restrictions on the parameters,
\begin{equation}\label{A}
\frac{1}{q}<\frac{N+\alpha}{2Np}<\frac{1}{2}.
\end{equation}
Then using the Hardy--Littlewood--Sobolev (HLS) and H\"older's inequalities we can control the Coulomb term, i.e.,
\begin{equation}\label{eq19}
\D_\alpha(|u|^p,|u|^p)\leq \mathscr{C}_{N,\alpha}\|u\|^{2p}_{\frac{2Np}{N+\alpha}}\leq \mathscr{C}_{N,\alpha}\|u\|^{2p\theta}_2\|u\|^{2p(1-\theta)}_{q}.
\end{equation}
Here $\mathscr{C}_{N,\alpha}=\frac{\Gamma\left((N-\alpha)/2\right)}{2^\alpha\pi^{\alpha/2}\Gamma((N+\alpha)/2)}\big(\frac{\Gamma(N)}{\Gamma(N/2)}\big)^{\alpha/N}$ is the sharp constant in the Hardy--Littlewood--Sobolev inequality \citelist{\cite{Lieb}\cite{Lieb-Loss}},
and $\theta \in (0,1)$ satisfies the condition
\begin{equation}\label{eq20}
\frac{N+\alpha}{2Np}=\frac{\theta}{2}+\frac{1-\theta}{q}, \quad \text{or } \ \theta=\frac{(N+\alpha)q-2Np}{Np(q-2)}.
\end{equation}
Therefore, the conditions in  \eqref{A} ensure that the energy $E$ is continuous and Fr\'echet differentiable on $L^2(\R^N)\cap L^q(\R^N)$, and its critical
points are solutions of \eqref{eqTF}. Moreover, critical points of $E$ (and solutions of \eqref{eqTF}) satisfy the Nehari identity
\begin{equation}\label{Nehari}
	\int_{\R^N}|u|^2dx+\int_{\R^N}|u|^qdx=\D_\alpha(|u|^p, |u|^p).
\end{equation}
By a {\em ground state} of \eqref{eqTF} we understand a nonnegative solution $u \in L^2(\R^N)\cap L^q(\R^N)$ of \eqref{eqTF} which has a minimal energy $E$ amongst all functions in the Poho\v zaev manifold $\mathscr{P}$, defined as 
\begin{equation}\label{P-manifold}
\mathscr{P}=\left\{u\neq 0: u\in L^2(\R^N)\cap L^q(\R^N),\; \mathcal{P}(u)=0\right\},
\end{equation}
where the functional $\mathcal{P}$ is given by
\begin{equation*}
	\mathcal{P}(u):=\frac{N}{2}\int_{\R^N}|u|^2dx+\frac{N}{q}\int_{\R^N}|u|^qdx-\frac{N+\alpha}{2p}\D_\alpha(|u|^p, |u|^p).
\end{equation*}
Since the energy $E$ is not bounded from below (by replacing $u$ with $u(\cdot/\lambda)$ with $\lambda \to +\infty$), 
constrained minimization techniques are better suited for the construction of ground states.
Moreover, the Poho\v zaev manifold $\mathscr{P}$ is preferred over the Nehari manifold characterised by  Eq.~\eqref{Nehari}, primarily because of 
simplifications due to the common expressions  $\frac{1}{2}\|u\|_2^2 + \frac{1}{q}\|u\|_q^q$
appearing in both $E(u)$ and $\mathcal{P}(u)$, as demonstrated in Section~\ref{sec-limit}.

Another way to construct ground states of \eqref{eqTF} is to look for maximizers of the 
Gagliardo--Nirenberg quotient associated to the interpolation inequality \eqref{eq19}, i.e.,  
\begin{equation}\label{eqCi}
  \mathscr{C}_{N,\alpha,p,q}:=\sup\left\{\frac{\D_\alpha(|v|^p,|v|^p)}{\|v\|_2^{2p\theta}\|v\|_q^{2p(1-\theta)}}:v\in
  L^2(\R^N)\cap L^q(\R^N), v\neq 0 \right\}.
\end{equation}
From ~\eqref{eq19}, it is clear that $\mathscr{C}_{N,\alpha,p,q}\le \mathscr{C}_{N,\alpha}$. 
Note that the quotient in Eq. \eqref{eqCi} is invariant w.r.t. translation, dilation and scaling;  
every maximizer for $\mathscr{C}_{N,\alpha,p,q}$ (if it exists) can be rescaled to a ground states solutions of \eqref{eqTF} (see Lemma \ref{rescaling_1} below).

Using symmetric rearrangements, Strauss' radial bounds and Helly's selection principle for radial functions, we prove the following result.

\begin{theorem}\label{t-GN}
Let $N\ge 1$, $\alpha\in(0,N)$, $p>\frac{N+\alpha}{N}$ and $q>\frac{2Np}{N+\alpha}$.
Then there exists a nonnegative radial nonincreasing maximizer $u_*\in L^2(\R^N)\cap L^q(\R^N)$ for $\mathscr{C}_{N,\alpha,p,q}$, which is also a ground state of the Thomas--Fermi equation \eqref{eqTF}.
\end{theorem}

\begin{remark} \label{rem:exact}
While the precise value  of $\mathscr{C}_{N,\alpha,p,q}$ is not known in general, 
we prove below (see Proposition \ref{limit_alpha}) that for fixed admissible values of $N$, $p$ and $q$, 
\begin{equation}\label{eq-0N}
\lim_{\alpha\to 0} \mathscr{C}_{N,\alpha,p,q}=1,\qquad \lim_{\alpha\to N} A^{-1}_{\alpha}\mathscr{C}_{N,\alpha,p,q}=1,
\end{equation}
here $A_\alpha$ is the Riesz constant. For specific combination of parameters, we can estimate $\mathscr{C}_{N,\alpha,p,q}$ by looking  at the ansatz $v(x)=\lambda (1+|x|^2/\mu^2)^{-\gamma}$. 
The invariance of the quotient \eqref{eqCi} with respect to $\lambda$ and $\mu$ suggests the choice
$\lambda=\mu=1$ for simplicity, leading to the governing equation 
\begin{equation}\label{eq:exactNval} 
c_1v+c_2|v|^{q-2}v=(I_\alpha *|v|^p)|v|^{p-2}v,
\end{equation}
with some $c_1,c_2>0$, that is equivalent to \eqref{eqTF} after a scaling.  Using the explicit representation
\[
I_{\alpha} *(1+|x|^2)^{-\gamma p} = \frac{\Gamma(\gamma p-\alpha/2)\Gamma((N-\alpha)/2)}{2^\alpha\Gamma(\gamma p)\Gamma(N/2)}\, {}_2F_1\Big(
\gamma p-\frac{\alpha}{2},\frac{N-\alpha}{2};\frac{N}{2};-|x|^2
\Big)
\] 
in terms of the Gauss hypergeometric function derived from~\cite{YanghongPME}, 
we can show that a nontrivial solution of \eqref{eq:exactNval} corresponds to the function $v(x) = (1+|x|^2)^{-(N+1)/2}$ with~\footnote{A symbolic solver in Maple or Mathematica is recommended. The same $v$, after appropriate rescaling, satisfies~\eqref{eqTF} 
with another set of parameters $p = \frac{N+\alpha+2}{N+2}, q= \frac{2N+2}{N+2}$, but the 
corresponding quotient in~\eqref{eqCi} is not bounded from above because the condition $p>(N+\alpha)/N$ in Theorem~\eqref{t-GN} is violated.} 
\begin{equation}\label{eq:explictpq}
p = \frac{N+\alpha+2}{N+1},\qquad q= \frac{2(N+2)}{N+1}.
\end{equation}
In the absence of uniqueness results for \eqref{eq:exactNval} we can not conclude that this solution is an optimizer for \eqref{eqCi}
however the optimal constant can be estimated (more details in Appendix~\ref{app:int}) as
\begin{equation}\label{C-estimate}
\mathscr{C}_{N,\alpha,p,q} \ge
 \frac{N(N+\alpha+2)}{\pi^{\alpha/2}2^{\alpha+1}(N+2)} \frac{\Gamma((N-\alpha)/2)}{\Gamma((N+\alpha)/2+1)}
  \left(
    \frac{N+2}{2(N+1)}\frac{\Gamma(N+1)}{\Gamma(N/2+1)}
  \right)^{\alpha/N}.
\end{equation}
We conjecture that $v(x) = (1+|x|^2)^{-(N+1)/2}$ is an optimizer for $\mathscr{C}_{N,\alpha,p,q}$ for the specific values of $p$ and $q$ in \eqref{eq:explictpq}, and \eqref{C-estimate} is actually an equality.

\end{remark}

\begin{remark}\label{r-sharp}
  The substitution $\rho=|u|^p$, $m=q/p$ and $n=2/p$ leads to an equivalent formulation of the quotient in Eq.
  \eqref{eqCi}, i.e.,
\begin{equation*}
\mathscr{C}_{N,\alpha,m,n}:=\sup\left\{\frac{\D_\alpha(\rho,\rho)}{\left(\int_{\R^N}\rho^{n} dx\right)^{\frac{2\theta}{n}}\left(\int_{\R^N}\rho^{m}
dx\right)^{\frac{2(1-\theta)}{m}}}:0\le\rho\in L^{n}(\R^N)\cap L^{m}(\R^N), \rho\neq 0 \right\}.
\end{equation*}
The corresponding interpolation inequality then takes the form
\begin{equation}\label{eq-Keller-S}
\iint_{\mathbb{R}^N \times \mathbb{R}^N} \frac{|\rho(x)|\,|\rho(y)|}{|x - y|^{d-\alpha}} dx\, dy\le C_{N,\alpha,m,n}\Big(\int_{{\mathbb R}^N}|\rho|^n dx\Big)^{\frac{2\theta}{n}} \Big(\int_{{\mathbb R}^N}|\rho|^m dx\Big)^{\frac{2(1-\theta)}{m}},
\end{equation}
for all $\rho \in L^{n}(\R^N)\cap L^{m}(\R^N)$, where $0<n<\frac{2N}{N+\alpha}<m$, and where by $L^n(\R^N)$ we denote the class of $n$--integrable functions with any $n>0$.
\eqref{eq-Keller-S} can be seen as a standard interpolation associated to the Hardy--Littlewood--Sobolev inequality,
which however includes {\em sublinear} exponents $n<1$.
Relevant variational problems with $n=1$ can be found in the early works by Auchmuty and Beals \cite{Auchmuty-Beals} and P.-L.~Lions \citelist{\cite{Lions-81}\cite{Lions-I1}*{Section II}} with $\alpha=2$. The case $\alpha\in(0,N)$ in the context of diffusion--aggregation models was studied in the form \eqref{eq-Keller-S} in the recent papers \citelist{\cite{Carrillo-CalcVar}\cite{Carrillo-NA}\cite{Carrillo-Sugiyama}\cite{Volzone}\cite{Carrillo-NA-2021}}.
We are not aware of any works where the case $n\neq 1$ was considered.
\end{remark}

Mathematically, the most striking phenomenon related to \eqref{eqTF} is how the behaviour
of the support of ground states to \eqref{eqTF} depends on the value of $p$.  Our main result in this work is the following.

\begin{theorem}\label{thm01}
  Let $N\ge 1$, $\alpha\in(0,N)$, $p>\frac{N+\alpha}{N}$, $q>\frac{2Np}{N+\alpha}$. Then
  every nonnegative radial nonincreasing ground state $u\in L^2(\R^N)\cap L^q(\R^N)$ of \eqref{eqTF} is $C^\infty$ in the set $\{x\in\R^N:u(x)>0\}$ and:
	\begin{itemize}
	\item[$(a)$]
    if $p<2$, then $\left\{u>0\right\}=\R^N$ and
    $\lim_{x\to\infty}u(x)|x|^\frac{N-\alpha}{2-p}=\left(A_\alpha\int_{\R^N}
    u^p\,dx\right)^{\frac1{2-p}}$, where $A_\alpha$ is the Riesz constant;
	\smallskip
	
	\item[$(b)$]
    if $p=2$, then $u:\R^N\to \R$ is a H\"older continuous and $\left\{u>0\right\}=B_R$ for some $R>0$;
	\smallskip
	
	\item[$(c)$]
    if $p>2$, then $\left\{u>0\right\}=B_R$ for some $R>0$ and $$u=\lambda\chi_{B_R}+\phi,$$  where $\lambda\ge \left(\frac{p-2}{q-p}\right)^{1/(q-2)}$ is a constant, and where $\phi:B_R\to\R$ is a H\"older continuous
    radial nonincreasing function such that $\phi(0)>0$ and $\lim_{|x|\to R}\phi(|x|)=0$.
	\end{itemize}
\end{theorem}

The case $p=2$ of Theorem \ref{thm01} is well studied.
The existence and qualitative properties of the ground states of \eqref{eqTF} in the case $N=3$, $\alpha=2$, $p=2$ and $q>8/3$ is classical and goes back to \citelist{\cite{Auchmuty-Beals}\cite{Lions-81}}. The case $N\ge 2$, $\alpha\in(0,N)$ and $q\ge q_c:=2(2-\alpha/N)$ is a recent study by J.~ Carrillo et al.~\citelist{\cite{Carrillo-CalcVar}\cite{Carrillo-NA}} in the context of Keller--Siegel systems. The case $p=2$ and $q>\frac{4N}{N+\alpha}$ appeared in \cite{ZLVMeps}*{Theorem 2.6}. 
In some special cases, it is known that the a bounded radially nonincreasing ground state of \eqref{eqTF} is unique. For $p=2$ and $\alpha=2$ this follows from \cite[Lemma 5]{Flucher}.  For $p=2$ and $\alpha<2$ this is the recent result in \cite[Theorem 1.1 and Proposition 5.4]{Volzone} (see also
\citelist{\cite{Carrillo-Sugiyama}\cite{Carrillo-NA-2021}\cite{Yao-Yao}} and further references therein).
For $p\neq2$, or for $p=2$ and $\alpha>2$, the uniqueness seems to be open at present.

In the case $N=3$, $\alpha=2$, $p=2$ and $q=4$ the (unique) nonnegative radial ground state of \eqref{eqTF}
is known explicitly and is given by the function
\begin{equation}\label{eq-v03}
	u(x)=\chi_{B_\pi}(x)\sqrt{\tfrac{\sin(|x|)}{|x|}}.
\end{equation}
This is (up to the physical constants) the Thomas--Fermi approximation solution for self--gravitating BEC observed in \citelist{\cite{Wang}\cite{Bohmer-Harko}\cite{Chavanis-11}} and the support radius $R=\pi$ is the approximate radius of the BEC star.  Note that $u\not\in H^1(\R^3)$.
For $p\ge 2$ and general values of $N$, $\alpha$ and $q$ the radius of the support of a ground state of \eqref{eqTF} can be easily estimated. In particular, in  Corollaries \ref{cor-supp-0} and \ref{cor-supp-N} we show that for fixed admissible $N$, $p$ and $q$, the radius of the support of the ground states diverges to $+\infty$ as $\alpha\to 0$, and  shrinks to zero as $\alpha\to N$. In Lemma \ref{lem34} and \ref{regu_p_big_12} we obtain quantitative estimates on the H\"older continuity of the ground state near the boundary of the support when $p\ge 2$.

\subsection{Thomas--Fermi limit for the Choquard equation \eqref{eqPeps}}
Next we prove that in the  relevant asymptotic regimes, ground states $w$ of \eqref{eqPeps}
described in Theorem \ref{thm01eps} converge, after the rescaling \eqref{TF-scale},
towards a ground state of the Thomas--Fermi equation \eqref{eqTF}. To identify the asymptotic regimes, observe that
the rescaling \eqref{TF-scale} transforms the Choquard energy $\mathcal{I}_{\eps}(u)$ in such a way that
$$\mathcal{J}_{\eps}(v)=\eps^{ \frac{q(N+\alpha)-2Np}{\alpha(q-2)}}\mathcal{I}_{\eps}(u),$$
where we denote
\begin{multline}\label{eq27}
	\mathcal J_{\eps}(v):=\frac12\eps^{\frac{2(2p+\alpha)-q(2+\alpha)}{\alpha(q-2)}}\int_{\R^N}|\nabla v|^2dx+\frac12\int_{\R^N}|v|^2dx+\frac{1}{q}\int_{\R^N}|v|^qdx-\frac{1}{2p}\D_\alpha(|v|^p,|v|^p).
\end{multline}
We note that if $\eps\to 0$ and
$q<2\frac{2p+\alpha}{2+\alpha}$, or if $\eps\to\infty$ and $q>2\frac{2p+\alpha}{2+\alpha}$
then $\eps^{\frac{2(2p+\alpha)-q(2+\alpha)}{\alpha(q-2)}}\to 0$, and {\em formally} the limit energy for $\mathcal J_{\eps}$ coincides with the Thomas--Fermi energy $E$. Combined with the
existence range of the ground state of \eqref{eqPeps} in Theorem \ref{thm01eps}, this
formally identifies the Thomas--Fermi limit regimes (see Figure \ref{fig:M1}). In Section \ref{sec-limit} we prove
the following result, that confirms our reasoning based on formal asymptotics and that covers the ranges
of $\alpha\neq 2$ and $p\neq 2$ left missing in \cite[Theorem 2.7 and 3.2]{ZLVMeps}.

\begin{theorem}\label{t-TF-0}
	Let $N\ge 3$ and $\alpha\in(0,N)$. Assume that either of the following holds:
	\begin{itemize}
    \item[$(i)$] $\frac{N+\alpha}{N}<p<\frac{N+\alpha}{N-2}$ and $q>2\frac{2p+\alpha}{2+\alpha}$, or $p>\frac{N+\alpha}{N-2}$ and $q>\frac{2Np}{N+\alpha}$;\smallskip
    \item[$(ii)$] $\frac{N+\alpha}{N}<p<\frac{N+\alpha}{N-2}$ and $\frac{2Np}{N+\alpha}<q<2\frac{2p+\alpha}{2+\alpha}$.
	\end{itemize}
	Then there exists a sequence  of parameters $(\eps_k)_{k\in\N}$ with the corresponding ground states $(u_{\varepsilon_k})$ of $(P_{\varepsilon_k})$ such that 
	\begin{center}
	$\eps_k\to\infty$ if $(i)$ holds, or $\eps_k\to 0$ if $(ii)$ holds,
	\end{center}
	and the rescaled sequence of ground states of \eqref{eqPeps}
	\begin{equation*}
		u_{\eps_k}(x):=\eps_k^{-\frac{1}{q-2}}w_{\eps_k}\big(\eps_k^{-\frac{2p-q}{\alpha(q-2)}}x\big)
	\end{equation*}
	converges in $L^2(\R^N)\cap L^q(\R^N)$ to a nonnegative ground state of the Thomas-Fermi equation \eqref{eqTF}. 
	Moreover, 
	$\varepsilon_k^{\frac{2(2p+\alpha)-q(2+\alpha)}{\alpha(q-2)}} \|\nabla u_{\varepsilon_k}\|^{2}_2\rightarrow 0$ as $k\to\infty$.
\end{theorem} 

Location of limit regimes $(i)$ and $(ii)$ on the $(p,q)$ plane is outlined in Figure \ref{fig:M1}.
Note that typically, the limit ground state of \eqref{eqTF} is not in $H^1(\R^N)$ and the quantity $\|\nabla u_{\varepsilon_k}\|^{2}_2$ in Theorem \ref{t-TF-0} blows up as $k\to\infty$. The qualitative properties and H\"older regularity of the ground state of \eqref{eqTF} established in Theorem \ref{thm01} become crucial in the analysis of Thomas--Fermi convergence in Theorem \ref{t-TF-0}.

\begin{figure}[t]
	\centering

	\tikzset{every picture/.style={line width=0.75pt}} 
	
	\begin{tikzpicture}[x=0.75pt,y=0.75pt,yscale=-0.75,xscale=0.75]
		
		
		
		\tikzset{
			pattern size/.store in=\mcSize,
			pattern size = 5pt,
			pattern thickness/.store in=\mcThickness,
			pattern thickness = 0.3pt,
			pattern radius/.store in=\mcRadius,
			pattern radius = 1pt}
		\makeatletter
		\pgfutil@ifundefined{pgf@pattern@name@_ro2tk9cc0}{
			\pgfdeclarepatternformonly[\mcThickness,\mcSize]{_ro2tk9cc0}
			{\pgfqpoint{0pt}{-\mcThickness}}
			{\pgfpoint{\mcSize}{\mcSize}}
			{\pgfpoint{\mcSize}{\mcSize}}
			{
				\pgfsetcolor{\tikz@pattern@color}
				\pgfsetlinewidth{\mcThickness}
				\pgfpathmoveto{\pgfqpoint{0pt}{\mcSize}}
				\pgfpathlineto{\pgfpoint{\mcSize+\mcThickness}{-\mcThickness}}
				\pgfusepath{stroke}
		}}
		\makeatother
		\tikzset{every picture/.style={line width=0.75pt}} 

		\draw[pattern={Lines[angle=45,distance=3pt,line width=0.5pt]}, pattern color=lightgray] 
		(310,310) -- (150,470) -- (150,130) -- (460,130)-- cycle;
		
		
		\draw[pattern={Lines[angle=135,distance=2pt, line width=0.4pt]}, pattern color=lightgray] (310,310) -- (150,510) -- (150,470)-- cycle;


		\draw [line width=1.5]  (71.07,510) -- (476.43,510)(110,112.97) -- (110,551.94) (469.43,505) -- (476.43,510) -- (469.43,515) (105,119.97) -- (110,112.97) -- (115,119.97)  ;


		\draw [color={rgb, 255:red, 126; green, 211; blue, 33 }  ,draw opacity=1 ][line width=1.0] (110,510) -- (460,160) ;
		

		\draw [color={rgb, 255:red, 208; green, 2; blue, 27 }  ,draw opacity=1 ][fill={rgb, 255:red, 234; green, 178; blue, 85 }  ,fill opacity=1 ][line width=1]    (150,130) -- (150,510) ;


		\draw [color={rgb, 255:red, 208; green, 2; blue, 27 }  ,draw opacity=1 ][line width=1]    (310,310) -- (310,510) ;

		%
		
		

		
		
		\draw [color={rgb, 255:red, 208; green, 2; blue, 27 }  ,draw opacity=1 ][line width=1]    (150,510) -- (310,310) ;
		%
		
		\draw [color={rgb, 255:red, 208; green, 2; blue, 27 }  ,draw opacity=1 ][line width=1]    (460,130) -- (310,310) ;
		\draw [shift={(310,310)}, rotate = 129.81] [color={rgb, 255:red, 208; green, 2; blue, 27 }  ,draw opacity=1 ][fill={rgb, 255:red, 208; green, 2; blue, 27 }  ,fill opacity=1 ][line width=1]      (0, 0) circle [x radius= 3.35, y radius= 3.35]   ;
		

		\draw (145,525) node [scale=0.9]  {$\frac{N+\alpha }{N}$};
		\draw (100,111) node [scale=0.9]  {$q$};
		\draw (467.5,526) node [scale=0.9]  {$p$};
		\draw (117.5,519) node [scale=0.9]  {$1$};
		\draw (102.5,499) node [scale=0.9]  {$2$};
		\draw (305,525) node [scale=0.9]  {$\frac{N+\alpha }{N-2}$};
		\draw  [fill={rgb, 255:red, 74; green, 144; blue, 226 }  ,fill opacity=1 ]  (300.5, 201.5) circle [x radius= 18.44, y radius= 18.26]   ;
		\draw (300.5,201.5) node [scale=0.9]  {$i$};
		\draw (510,155) node [scale=0.9]  {$q=2\frac{2p+\alpha }{2+\alpha}$};
		\draw (505,125) node [scale=0.9]  {$q=\frac{2Np}{N+\alpha}$};
		
		\draw (527,428) node [scale=0.6] [align=left] {\huge{$\boxed{(ii):\eps\to 0}$}};
		\draw (527,350) node [scale=0.6] [align=left] {\huge{$\boxed{(i):\eps\to \infty}$}};

		\draw  [fill={rgb, 255:red, 248; green, 231; blue, 28 }  ,fill opacity=1 ][line width=0.75]   (168, 469) circle [x radius= 13, y radius= 11]   ;
		\draw (168,469) node [scale=0.9]  {$ii$};


	\end{tikzpicture}
	
\caption{Thomas--Fermi limit regimes $(i)$ and $(ii)$ in Theorem~\ref{t-TF-0}} \label{fig:M1}
\end{figure}
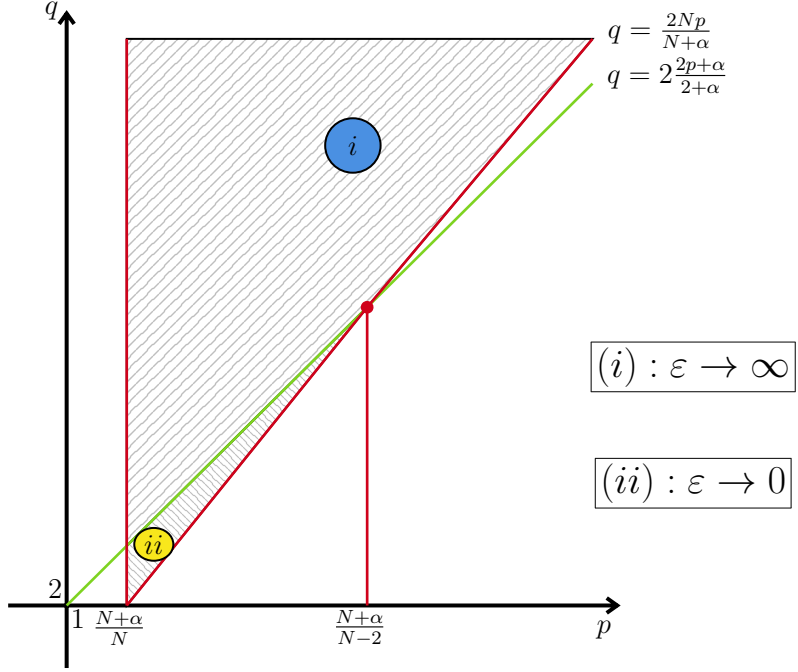

\subsection*{Notations}\label{s-notation}

For real valued nonnegative functions $f(t), g(t)$ defined on a subset of $\R_+$, we write:
\smallskip

$f(t)\lesssim g(t)$ if there exists $C>0$ independent of $t$
such that $f(t) \le C g(t)$;

$f(t)\gtrsim g(t)$ if $g(t)\lesssim f(t)$;

$f(t)\sim g(t)$ if $f(t)\lesssim g(t)$ and $f(t)\gtrsim g(t)$.


\smallskip
\noindent
Bearing in mind that $f(t),g(t) \geq 0$, we write $f(t)=\O(g(t))$ if $f(t)\sim g(t)$, and $f(t)=o(g(t))$ if $\lim\frac{f(t)}{g(t)}=0$.
As usual, $B_R=\{x\in\R^N:|x|<R\}$ and $|B_R|$ is the volume of $B_R$. By $C,c,c_1$ etc. we denote generic positive constants whose value may change from line to line.

\section{Proof of Theorem \ref{t-GN}}\label{sec04}

In what follows, unless specified otherwise, we always assume that $N\ge 1$ and $\alpha\in(0,N)$.

\subsection{Existence of an optimiser}
For $u\in L^q(\R^N)\cap L^2(\R^N)$, denote the quotient
$$
\mathcal{R}_\alpha(u):=\frac{\D_\alpha(|u|^p,|u|^p)}{\|u\|_2^{2p\theta}\|u\|_q^{2p(1-\theta)}} .
$$
We are going to show that the best constant
\begin{equation}\label{eq16}
  \mathscr{C}_{N,\alpha,p,q}=\sup\Big\{\mathcal{R}_\alpha(u): u\in L^q(\R^N)\cap L^2(\R^N), u\neq 0\Big\}
\end{equation}
is achieved, following(with some modifications the arguments in \cite{Frank}*{Proposition 8}.

\begin{lemma}\label{t-GN-existence}
	Let  $p>\frac{N+\alpha}{N}$ and $q>\frac{2Np}{N+\alpha}$.
	Then there exists a nonnegative radial nonincreasing maximizer $u\in L^2(\R^N)\cap L^q(\R^N)$ for $\mathscr{C}_{N,\alpha,p,q}$.
\end{lemma}

\proof
Let  $(u_n)\subset L^2(\R^N)\cap L^q(\R^N)$ be a maximizing sequence, such that $\mathcal{R}_\alpha(u_n)\to \C_{N,\alpha,p,q}$ as $n\to \infty$.
Let $u^*_n$ denote the Schwartz spherical rearrangements of $|u_n|$.
Then $u^*_n$ is nonnegative radially symmetric nonincreasing, and
	\begin{equation}\label{eq32}
	\| u_n\|_2^2=\| u^*_n\|_2^2, \quad \| u_n\|_q^q=\| u^*_n\|_q^q, \quad \D_\alpha(|u_n|^p,|u_n|^p)\leq \D_\alpha((u_n^*)^p,(u_n^*)^p).
	\end{equation}
	Therefore $\mathcal{R}_\alpha(u_n)\leq \mathcal{R}_\alpha(u^*_n)$ and $(u_n^*)$ is also a maximizing sequence of $\C_{N,\alpha,p,q}$.
	Without loss of generality, we can denote $u_n^*$ by $u_n$ in the rest of the proof.
	
	By using the scaling invariance and homogeneity of  $\mathcal{R}_\alpha$ we can assume that $\|u_n\|_2=\|u_n\|_q=1$, so that $$\mathcal{R}_\alpha(u_n)=\D_\alpha(u_n^p,u_n^p)\to \C_{N,\alpha,p,q}$$
	as $n\to \infty$.  Using Strauss' $L^s$--bounds \cite{Strauss} with $s=2$ and $s=q$, we conclude that
	$$u_n(x)\leq U(x):=C\min\{|x|^{-N/2}, |x|^{-N/q}\}.$$

	Since $U\in L^s(\bbr^N)$ for $s\in(2,q)$, by Helly's selection principle there exists
  $0\leq u\in L^2(\R^N)\cap L^q(\R^N)$ such that $u_n(x)\to u(x)$ a.e.~in $\bbr^N$ as $n\to \infty$.  By the Lebesgue dominated convergence theorem, we see that for $s\in (2, q)$,
	$$
	\lim_{n\to\infty}\int_{\bbr^N}|u_n|^sdx=\int_{\bbr^N}|u|^sdx.
	$$
From the restriction $q>\frac{2Np}{N+\alpha}>p$, by the nonlocal Brezis--Lieb Lemma \cite{MMS16}*{Proposition 4.7}, we conclude that
	$$
	\lim_{n\to\infty}\D_\alpha(u_n^p,u_n^p)=\D_\alpha(u^p,u^p)=\C_{N,\alpha,p,q}.
	$$
	By Fatou's Lemma, we get that
	$$
	1=\lim_{n\to\infty}\int_{\bbr^N}|u_n|^2dx\geq \int_{\bbr^N}|u|^2dx>0, \quad  1=\lim_{n\to\infty}\int_{\bbr^N}|u_n|^qdx\geq \int_{\bbr^N}|u|^qdx>0.
	$$
	We claim that $\int_{\bbr^N}|u|^2dx=\int_{\bbr^N}|u|^qdx=1$.   Otherwise, if  $\|u\|_2\|u\|_q<1$, then
	$$
	\C_{N,\alpha,p,q}\geq \mathcal{R}_\alpha(u) >\D_\alpha(u^p,u^p)=\lim_{n\to\infty}\D_\alpha(u_n^p,u_n^p)=\lim_{n\to\infty}\mathcal{R}_\alpha(u_n)=\C_{N,\alpha,p,q},
	$$
	a contradiction.  Therefore, our claim holds and
	$$
	\lim_{n\to\infty}\int_{\bbr^N}|u_n|^2dx=\int_{\bbr^N}|u|^2dx, \quad \lim_{n\to\infty}\int_{\bbr^N}|u_n|^qdx=\int_{\bbr^N}|u|^qdx,
	$$
	that is, $u_n$ converges to $u$ strongly in $L^s(\bbr^N)$ for $s\in [2, q]$, and $\C_{N,\alpha,p,q}=\mathcal{R}_\alpha(u)$.  
\qed

Next we briefly discuss the asymptotic behaviours of the optimal constant $\C_{N,\alpha,p,q}$ when $\alpha$ approaches $0$ or $N$. 
\begin{proposition}\label{limit_alpha}
	Assume that $N\ge 1$, $\alpha\in(0,N)$.
	If $p>1$ and $q\ge 2p$ then
	\begin{equation}\label{lim_1}
		\lim_{\alpha\to 0} \mathscr{C}_{N,\alpha,p,q}=1. 
	\end{equation}
	Furthermore, if $p\ge 2$ and $q>p$ then
	\begin{equation}\label{lim_2}
		\lim_{\alpha\to N} A^{-1}_{\alpha}\mathscr{C}_{N,\alpha,p,q}=1.
	\end{equation}
\end{proposition}
\begin{proof}
	First of all we notice that $p>1$ implies $p>\frac{N+\alpha}{N}$ for every $\alpha$ sufficiently close to zero, and $2p>\frac{2Np}{N+\alpha}$ for every
  $\alpha\in (0,N)$. Similarly, the assumption $q>p$ ensures that $q>\frac{2Np}{N+\alpha}$ for every $\alpha$ sufficiently close to $N$, and clearly $p \geq
  2>\frac{N+\alpha}{N}$. Thus, under our assumptions on $p$ and $q$, if $\alpha$ is sufficiently close to zero or $N$, the optimal constant $\C_{N,\alpha,p,q}$ is well defined. First, we prove that 
  $$\limsup_{\alpha\to 0}\mathscr{C}_{N,\alpha,p,q}\le 1.$$
	By the HLS inequality and standard properties of the Gamma function, we conclude that 
	\begin{equation}\label{su_sup}
		\mathscr{C}_{N,\alpha,p,q}\le  \mathscr{C}_{N,\alpha}=(2\sqrt{\pi})^{-\alpha}\frac{\Gamma\left(\frac{N-\alpha}{2}\right)}{\Gamma\left(\frac{N+\alpha}{2}\right)}\left(\frac{\Gamma(N)}{\Gamma\left(\frac{N}{2}\right)}\right)^{\frac{\alpha}{N}}\to 1\quad \text{as}\ \alpha\to 0,
	\end{equation}
  where $\mathscr{C}_{N,\alpha}$ is the sharp constant in the HLS inequality \eqref{eq19}. 
	
  To derive a lower bound of $\mathscr{C}_{N,\alpha,p,q}$ we take $u=\chi_{B_1}$ as the trial function. Then, by the explicit expression    for the Riesz potential of a characteristic function given in Eq. \eqref{riesz_13bis} below, 
	\begin{equation}\label{lower_sup}
		\begin{split}
			\mathscr{C}_{N,\alpha,p,q}\ge |B_1|^{-p\theta-\frac{2p(1-\theta)}{q}} \frac{\Gamma((N-\alpha)/2)}{2^\alpha\Gamma(1+\alpha/2)\Gamma(N/2)}\int_{B_1}
			{}_{2}F_1(-\alpha/2,(N-\alpha)/2;N/2;|x|^2)dx,
		\end{split}
	\end{equation}
	where $\theta=\theta(\alpha)$ is defined in Eq. \eqref{eq20}.
	Next we note the following limits (with fixed $p, q$),
	\begin{equation*}
		\begin{split}
			&\lim_{\alpha\to 0} \left(p\theta(\alpha)+\frac{2p(1-\theta(\alpha))}{q}\right)=1;\\
      & \lim_{\alpha\to 0} 	{}_{2}F_1(-\alpha/2,(N-\alpha)/2;N/2;|x|^2)=1,\quad \text{for\ every}\ |x|<1,
		\end{split}
	\end{equation*}
	Thus Fatou's lemma yields
	\begin{equation*}
		\liminf_{\alpha\to 0} \mathscr{C}_{N,\alpha,p,q}\ge |B_1|^{-1}\int_{B_1}\liminf_{\alpha\to 0}	{}_{2}F_1(-\alpha/2,(N-\alpha)/2;N/2;|x|^2) dx=1,
	\end{equation*}
	which concludes the proof of \eqref{lim_1}.
	
	To derive the limit \eqref{lim_2}, we notice that 
	\begin{equation*}
		\begin{split}
			&\lim_{\alpha\to N} \left(p\theta(\alpha)+\frac{2p(1-\theta(\alpha))}{q}\right)=2;\\
      & \lim_{\alpha\to N} 	{}_{2}F_1(-\alpha/2,(N-\alpha)/2;N/2;|x|^2)=1,\quad \text{for\ every}\ |x|<1,
		\end{split}
	\end{equation*}
	with the same $\theta=\theta(\alpha)$  as in Eq. \eqref{eq20}.
	Hence, by Eq. \eqref{lower_sup} and Fatou's Lemma we deduce 
	\begin{equation*}
		\begin{split}
			\liminf_{\alpha\to N} A^{-1}_{\alpha}\mathscr{C}_{N,\alpha,p,q}\ge |B_1|^{-1} \pi ^{\frac{N}{2}}\liminf_{\alpha\to N}\frac{1}{\Gamma(1+\frac{\alpha}{2})}=|B_1|^{-1} \frac{\pi^{\frac{N}{2}}}{\Gamma(1+\frac{N}{2})}=1.
		\end{split}
	\end{equation*}
Finally, similarly \eqref{su_sup} and using the explicit form of $A_{\alpha}$ and $\mathscr{C}_{N,\alpha}$ we obtain
	$$\limsup_{\alpha\to N}  A^{-1}_{\alpha}\mathscr{C}_{N,\alpha,p,q}\leq
	\limsup_{\alpha\to N}  A^{-1}_{\alpha}\mathscr{C}_{N,\alpha} \leq 1,$$
	which concludes the proof of \eqref{lim_2}.
\end{proof}

In Section \ref{s-support} we will show  that if $p\ge 2$ then maximizers for $\mathscr{C}_{N,\alpha,p,q}$ have compact support. Estimates in Proposition \ref{limit_alpha} will be then used to estimate the radius of support of the maximizers as $\alpha\to 0$ and $\alpha\to N$.

\subsection{Maximizer has full support if $p<2$}
Our next observation is that in the case $p<2$ maximizers for $\mathscr{C}_{N,\alpha,p,q}$ have full support $\R^N$.

\begin{lemma}\label{lem51}
	Assume that	$p>\frac{N+\alpha}{N}$ and $q>\frac{2Np}{N+\alpha}$.
	Let $u\in L^2(\R^N)\cap L^q(\R^N)$ be a nonnegative radial nonincreasing maximizer for $\mathscr{C}_{N,\alpha,p,q}$. 
	If $p<2$ then $\mathrm{Supp}(u)=\R^N$.
\end{lemma}
\begin{proof}
	Without loss of generality we can assume that $\|u\|_2=\|u\|_q=1$.  Arguing by contradiction, assume that there exists an open set $A\subset \bbrn$ with $A\cap \mathrm{Supp}(u)=\varnothing$ and $0<|A|<+\infty$. For $\epsilon>0$, consider the family of trial functions $v_\epsilon:=u+\epsilon\chi_{A}\in L^2(\R^N)\cap L^q(\R^N)$. We obtain
	$$\aligned
	\mathcal{R}_\alpha(v_\epsilon)
	&=\frac{\D_\alpha(|u+\epsilon\chi_{A}|^p, |u+\epsilon\chi_{A}|^p)}{\Big(\int_{\bbrn}(u^2+\epsilon^2\chi_{A})dx\Big)^{p\theta}\Big(\int_{\bbrn}(u^q+\epsilon^q\chi_A)dx\Big)^{\frac{2p(1-\theta)}{q}}}\\
	&\geq\frac{\D_\alpha(|u|^p, |u|^p)+2\epsilon^p \D_\alpha(|u|^p, \chi_A)+\epsilon^{2p}\D_\alpha(\chi_A, \chi_A)}{(1+p\theta\epsilon^{2}|A|)\big(1+\epsilon^q\frac{2p(1-\theta)}{q}|A|\big)}\\
	&\geq\frac{\D_\alpha(|u|^p, |u|^p)+2\epsilon^p \D_\alpha(|u|^p, \chi_A)+\epsilon^{2p}\D_\alpha(\chi_A, \chi_A)}{1+C\epsilon^{2}}\\
	&\geq \mathscr{C}_{N,\alpha,p,q}+\frac{2\epsilon^p \D_\alpha(|u|^p, \chi_A)+\epsilon^{2p}\D_\alpha(\chi_A, \chi_A)-C\epsilon^{2}}{1+C\epsilon^{2}}.
	\endaligned
	$$
	Because $p<2$, there exists $\epsilon_0>0$ such that for all $\epsilon\in (0, \epsilon_0)$, we have $\mathcal{R}_\alpha(v_\epsilon)>\mathscr{C}_{N,\alpha,p,q}$, which contradicts to the fact that $u$ is a maximizer.
\end{proof}

\subsection{Connection with ground states of \eqref{eqTF}}
A direct computation shows that the Euler-Lagrange equation of $\mathcal{R}_\alpha(u)$
(or equivalently $\log\mathcal{R}_\alpha(u)$) for $u\in L^q(\R^N)\cap L^2(\R^N)$ has the form
\begin{equation}\label{eq17}
A u
+B |u|^{q-2}u
= C (I_\alpha *|u|^p)|u|^{p-2}u
\quad \text{in $\R^N$},
\end{equation}
where
$$
A=\frac{2p\theta}{\|u\|_2^2}, \quad 
B=\frac{2p(1-\theta)}{\|u\|_q^q},\quad
C=\frac{2p}{\D_\alpha(|u|^{p},|u|^{p})}.
$$
In particular, maximizers of $\C_{N,\alpha,p,q}$ constructed in the proof
of Lemma \ref{t-GN-existence} are weak solutions of Eq. \eqref{eq17} and, after a rescaling, of  \eqref{eqTF}.  Indeed, given a maximizer $u$ for
$\C_{N,\alpha,p,q}$, for $\lambda,\mu>0$, consider the two-parameter family of functions $u_{\lambda,\mu}(x)=\lambda u(\mu x)$. In view of the scaling invariance and homogeneity of  $\mathcal{R}_\alpha$, we know that $\mathcal{R}_\alpha(u)=\mathcal{R}_\alpha(u_{\lambda,\mu})$. Therefore if we set $\lambda_{*}$ and $\mu_{*}$ such that
\begin{equation}\label{lamba_mu_opt}
	\lambda^{q-2}_{*}=\left(\frac{1-\theta}{\theta}\right) \frac{\left\|u\right\|^2_{2}}{\left\|u\right\|^q_q},\qquad \mu^{\alpha}_{*}=\left(\frac{1-\theta}{
		\lambda^{q-2p}_{*}} \right)\frac{\D_{\alpha}(|u|^p,|u|^p)}{\left\|u\right\|^q_{q}},
\end{equation}
we obtain $A=B=C$, and hence $u_{\lambda_*,\mu_*}$ is a solution of $(TF)$. In the next lemma we prove that $u_{\lambda_*,\mu_*}$ is a ground state of $(TF)$, i.e. $u_{\lambda_*,\mu_*}$ belongs to the Poho\v zaev manifold $\mathscr{P}$, defined in \eqref{P-manifold}, and $E(u_{\lambda_*,\mu_*})=\sigma_*$, where
\begin{equation}\label{def-sigma}
	\sigma_*: = \inf_{u\in \mathscr{P}} E(u)
\end{equation}
is the {\em ground state energy} of $(TF)$.

\begin{lemma}\label{rescaling_1}
	Assume that	$p>\frac{N+\alpha}{N}$ and $q>\frac{2Np}{N+\alpha}$.
	Let $u\in L^2(\R^N)\cap L^q(\R^N)$ be a nonnegative radial nonincreasing maximizer for $\mathscr{C}_{N,\alpha,p,q}$. 
	Then the function $u_{\lambda_*,\mu_*}$, where $\lambda_*$ and $\mu_*$ are defined by \eqref{lamba_mu_opt}, is a ground state of $(TF)$.
\end{lemma}
\begin{proof}
	For the brevity of notation, let's set $u_*=u_{\lambda_*,\mu_*}$ for any maximizer $u$ of the quotient $\mathcal{R}_\alpha$.  From \eqref{lamba_mu_opt} we deduce directly that 
		\[
		\frac{N}{2}\|u_*\|_2^2 +\frac{N}{q}\|u_*\|_q^q -\frac{N+\alpha}{2p}
		\D_{\alpha}(|u_*|^p,|u_*|^p) 
		=\frac{N}{\theta}\left(\frac{\theta}{2} + \frac{1-\theta}{q} 
		-\frac{N+\alpha}{2Np}\right) \|u_*\|_2^2=0,
		\]
		that is  $u_*\in\mathscr{P}$. 	
	We only need to prove that $E(u_*)=\sigma_*$, where $\sigma_*$ is the ground state energy defined in \eqref{def-sigma}. To this aim, we explore the relation between the optimal constant $\mathscr{C}_{N,\alpha,p,q}$ and the ground state energy $\sigma_*$.
	
	As an intermediate step, consider the functional 
	$$\mathcal{E}(w)=\frac{1}{2}\|w\|^2_2+\frac{1}{q}\|w\|^q_{q}$$
	on the set 
	$$\mathcal{A}:=\left\{w\in L^2(\R^N)\cap L^q(\R^N):\D_{\alpha}(|w|^p,|w|^p)=1 \right\}$$
	and note that $\mathcal{A}$ is invariant with respect to the rescaling
	$w_t(\cdot)=t^{-\frac{N+\alpha}{2p}}w(\cdot/t)$.
	
For a given $w\in \mathcal{A}$, by optimizing the quantity $\mathcal{E}(w_t)$ with
respect to $t$ (see \cite{PhD} for the details) we deduce that 
\[
  \mathcal{E}(w) \geq \mathcal{E}(w_{t^*}) 
  = \left(\|w\|_2^{2p\theta}\|w\|_q^{2p(1-\theta)}\right)^{\frac{N}{N+\alpha}} \theta_*,
\]
with 
\begin{equation}\label{theta-star}
  t^* = \left(
    \frac{(N+\alpha)q-2Np}{q(Np-N-\alpha)}\frac{\|w\|_q^q}{\|w\|_2^2}
  \right)^{\frac{2p}{(q-2)(N+\alpha)}},\quad 
	\theta_*=\left(\frac{1-\theta}{\theta}\right)^{\frac{q\theta}{2(1-\theta)+q\theta}}\left(\frac{N+\alpha}{2Np(1-\theta)}\right).
\end{equation}
As a consequence,
\begin{equation}
\inf_{w\in \mathcal{A}}\mathcal{E}(w)=\inf_{w\in \mathcal{A}} 
   \theta_*\left(\|w\|_2^{2p\theta}\|w\|_q^{2p(1-\theta)}\right)^{\frac{N}{N+\alpha}} 
=\theta_* \mathscr{C}^{-\frac{N}{N+\alpha}}_{N,\alpha,p,q}, 
\end{equation}
On the other hand, functions in $\mathscr{P}$ and $\mathcal{A}$ can be one--to--one mapped to each other with a spatial scaling. That is, 
for any $u \in \mathscr{P}$ we have $u(\cdot/\tau_u) \in \mathcal{A}$ for 
$\tau_u: = \big(\mathcal{D}_\alpha(|u|^p,|u|^p)\big)^{-1/(N+\alpha)}=\big(\frac{N+\alpha}{2Np\mathcal{E}(u)}\big)^{1/\alpha}$; and for any $w \in \mathcal{A}$ we have $w(\cdot/t_w) \in \mathscr{P}$ with 
$t_w: = \big(\frac{2Np\mathcal{E}(w)}{N+\alpha}\big)^{1/\alpha}$. Therefore, given $u \in \mathscr{P}$ and $w=u(\cdot/\tau_u)\in\mathcal A$, we have
\begin{equation}\label{rel_minimi-minus}
  E(u) = \mathcal{E}(u) - \frac{1}{2p}\D_\alpha(|u|^p,|u|^p)
  =\frac{\alpha}{N+\alpha}\left(\frac{2Np}{N+\alpha}\right)^{\frac{N}{\alpha}}\mathcal{E}(w)^{\frac{N+\alpha}{\alpha}}.
\end{equation}
Using the one-to-one correspondence between functions in $\mathscr{P}$ and $\mathcal{A}$ and taking the infimum in \eqref{rel_minimi-minus} we obtain (see \cite{PhD} for further details),
\begin{equation}\label{rel_minimi}
  \sigma_*  =\frac{\alpha}{N+\alpha}\left(\frac{2Np}{N+\alpha}\right)^{\frac{N}{\alpha}}\left(\inf_{w \in \mathcal{A}}\mathcal{E}(w)\right)^{\frac{N+\alpha}{\alpha}}
=\alpha (2Np)^{\frac{N}{\alpha}}\left(\frac{\theta_*}{N+\alpha}\right)^{\frac{N+\alpha}{\alpha}}\mathscr{C}^{-\frac{N}{\alpha}}_{N,\alpha,p,q},
\end{equation}
establishing the relation between these optimal values.

In order to conclude, we only need to show that the relation \eqref{rel_minimi} is satisfied  
for $u_*\in\mathscr{P}$. In fact, the choice of $\lambda_*$ and $\mu_*$ in \eqref{lamba_mu_opt} implies 
\[
  \|u_*\|_q^q = \frac{1-\theta}{\theta}\|u_*\|_2^2,\quad 
  \mathcal{D}_\alpha(|u_*|^p,|u_*|^p)
  =\frac{1}{\theta}\|u_*\|_2^2,
\]
which enables us to write both $E(u_*)$ and $\mathcal{R}_\alpha(u_*)$ in terms of $\|u_*\|_2^2$. That is, 
\[
  E(u_*) = 
	 \frac{\alpha (q-2)}{
		2(N+\alpha)q-4Np
  }\|u_*\|_2^2,\quad \mbox{ and  }\quad 
	\mathcal{R}_\alpha(u_*) 
	=\frac{1}{\theta}\left(\frac{\theta}{1-\theta}\right)^{2p(1-\theta)/q}
  \|u_*\|_2^{-2\alpha/N}.
\]
By eliminating $\|u_*\|_2$ from the above two relations, we infer that
$$E(u_*)=\alpha
(2Np)^{\frac{N}{\alpha}}\left(\frac{\theta_*}{N+\alpha}\right)^{\frac{N+\alpha}{\alpha}}\mathcal{R}_\alpha(u_*)^{-\frac{N}{\alpha}},$$
which in view of \eqref{rel_minimi} and $\mathcal{R}_\alpha(u_*)=\mathscr{C}_{N,\alpha,p,q}$ implies $E(u_*)=\sigma_*$. 
\end{proof}

\smallskip
\proof[Proof of Theorem \ref{t-GN}]
Follows from Lemma \ref{t-GN-existence} and \ref{rescaling_1}.
\qed

\section{Regularity, decay and support}\label{s-support}

In this section we establish qualitative properties of the ground states of $(TF)$, described in Theorem \ref{thm01} and, additionally, obtain a quantitative characterisation of the H\"older continuity of the ground states.

\subsection{Regularity, decay and support properties}
Recall that if \(s \in (1,\frac{N}{\alpha})\) and $\frac{1}{t}=\frac{1}{s}-\frac{\alpha}{N}$, then the HLS inequality implies that the operator
$$I_\alpha*(\cdot ): L^s (\R^N)\to L^t (\R^N)$$
is  bounded  \cite{Lieb}. We first establish the following fact about the far field behaviour of
$I_\alpha*u^p$: if the nonnegative function $u$ decays 
fast enough, then $I_\alpha *u^p$ decays algebraically like the Riesz potential $I_\alpha$ itself.

\begin{lemma}\label{lemm_11}
	Assume that	$p>\frac{N+\alpha}{N}$ and $q>\frac{2Np}{N+\alpha}$.
	Let $u\in L^2(\R^N)\cap L^q(\R^N)$ be a nonnegative radial nonincreasing solution of \eqref{eqTF}. Then
	there exists $\epsilon>0$ such that $u\in L^{p-\epsilon}(\R^N)$ and 
	\begin{equation}\label{l-potential-0}
		\lim_{|x|\to\infty}\frac{I_{\alpha}*u^p}{I_{\alpha}(x)\int_{\bbrn}u^pdx}=1.
	\end{equation}	
\end{lemma}
\begin{proof}
We first prove that $u \in L^{p-\epsilon}(\R^N)$ some $\epsilon>0$, which is trivial if $p>2$.  Otherwise if $p \in (\frac{N+\alpha}{N}, 2]$, 
we can show that $u \in L^{s_n}(\R^N)$ for a sequence $(s_n)$ of positive decreasing exponents eventually smaller than $p$.   

First by H\"older's inequality, we see that
	\begin{equation}\label{eq11}
		\int_{\bbrn}|(I_{\alpha}*u^p)u^{p-1}|^{\sigma}dx\leq \Big(\int_{\bbrn}|(I_{\alpha}*u^p)|^{\sigma t}dx\Big)^{\frac{1}{t}}\Big(\int_{\bbrn}u^{(p-1)\sigma r}dx\Big)^{\frac{1}{r}},
	\end{equation}
  provided that $1/t+1/r=1$ for positive $t$ and $r$. We want to find a sequence $(s_n)$ of positive numbers,
  so that if $u \in L^{s_n}(\R^N)$, then $u \in L^{s_{n+1}}(\R^N)$. By choosing the parameters $\sigma,t$
  and $r$ in Eq.~\eqref{eq11} so that 
  \[
    \sigma = s_{n+1},\qquad (p-1)\sigma r = s_n, \qquad 
    \frac{1}{\sigma t} = 
\frac{p}{s_n} - \frac{\alpha}{N}.    
  \]
The last equation, arising  from the HLS inequality,
has to be supplied with the condition $\alpha/N < p/s_n < 1$, or $s_n \in (p, Np/\alpha)$.  Therefore, the sequence $(s_n)$ satisfies the recursion relation
  \begin{equation}\label{eq12}
		\frac{1}{s_{n+1}}
= \frac{1}{\sigma t} + \frac{1}{\sigma r}
= \frac{p}{s_n} - \frac{\alpha}{N}+ \frac{p-1}{s_n}		
		=\frac{2p-1}{s_n}-\frac{\alpha}{N}.
	\end{equation}
With the (unstable) fixed point $s_* = 2N(p-1)/\alpha>2$ (as $p>(N+\alpha)/N$), 
the general term can be written as 
\begin{equation}\label{eq:rec}
\frac{1}{s_{n}} = (2p-1)^n\left(\frac{1}{s_0}-\frac{1}{s_*}\right)+ 
\frac{1}{s_*}.
\end{equation}
If $s_0$ is chosen to be any number inside the interval $(2,s_*)$, then $s_n$ is monotonically decreasing to zero. Therefore, we can look for the largest integer $n_0$ such that $s_{n_0}>p$. If $s_{n_0+1} < p$, then $u \in L^{p-\epsilon}$ for 
any positive $\epsilon < p-s_{n_0+1}$. Otherwise, if $s_{n_0+1}=p$, we can always choose 
$s_0$ slightly smaller to get $s_{n_0+1}<p$, since by the recursion relation 
Eq.~\eqref{eq:rec}, $s_n$ depends continuously and  monotonically on the initial value $s_0 \in (2,s_*)$.

Consequently, since $u$ is radially symmetric and nonincreasing, by the Strauss' $L^{p-\epsilon}$--bound we have
a faster decay estimate 
$$u(|x|)\leq C |x|^{-\frac{N}{p-\epsilon}}\quad(|x|>0).$$ 
Then, by \cite{TF}*{Lemma 6.1} we obtain the desired limit~\eqref{l-potential-0}. 
\end{proof}

\begin{corollary}\label{cor_12}
	Assume that $\frac{N+\alpha}{N}<p<2$ and $q>\frac{2Np}{N+\alpha}$. 
	Let $u\in L^2(\R^N)\cap L^q(\R^N)$ be a nonnegative radial nonincreasing solution of \eqref{eqTF}. Then $u$ satisfies the following algebraic decay rate
	\begin{equation} \label{eq:ufar}
	\lim_{x\to\infty}u(x)|x|^\frac{N-\alpha}{2-p}=\left(A_\alpha\int_{\R^N} u^p\,dx\right)^{\frac1{2-p}},
	\end{equation}
	where $A_\alpha>0$ is the Riesz constant.  Furthermore, we have that $u \in L^1(\R^N)$.
\end{corollary}
\begin{proof}
	By Lemma~\ref{lemm_11}, $I_{\alpha}*u^p=I_{\alpha}(x)\int_{\bbrn}u^pdx\big(1+o(1)\big)$ as $|x|\to\infty$. 
	Hence, the governing equation $(TF)$ implies that
	$$
	\lim_{|x|\to \infty}u(x)^{2-p}|x|^{N-\alpha}(1+u^{q-2}(x))
	=\lim_{|x|\to \infty} |x|^{N-\alpha}I_{\alpha}*u^p= A_\alpha\int_{\R^N} u^p\,dx.
	$$
	From the monotonicity of $u$ and the fact that $q>2$, we conclude that $u^{q-2}(|x|)$ vanishes at infinity, and therefore
	$$
	\lim_{|x|\to \infty}u^{2-p}(x)|x|^{N-\alpha}=A_\alpha\int_{\R^N} u^p\,dx,
	$$
	which is equivalent to~\eqref{eq:ufar}. From the condition $\frac{N+\alpha}{N}<p<2$, 
	the power $\frac{N-\alpha}{2-p}$ is strictly larger than $N$. That is, $u$ decays faster than $|x|^{-N}$ and hence 
	$u \in L^1(\R^N)$.
\end{proof}

\begin{lemma}\label{lem34}
	Assume that	$p>\frac{N+\alpha}{N}$ and $q>\frac{2Np}{N+\alpha}$.
	Let $u\in L^2(\R^N)\cap L^q(\R^N)$ be a nonnegative solution of \eqref{eqTF}. Then   $u\in L^\infty(\R^N)$ and
	\begin{equation}\label{eqLinfty}
		I_{\alpha}*u^p\in C^{0,\tau}(\R^N)\quad \text{for every}\quad  \tau\in (0, \min\{\alpha, 1\})
	\end{equation} 
	In particular, if $p\le 2$ the following hold:
	\begin{itemize}
		\item[$(i)$] If $p<2$ then $u$ is H\"older continuous in $\left\{u>0\right\}$ of order $\tau$, for every $\tau\in (0, \min\{\alpha, 1\})$.
		\smallskip			
		\item[$(ii)$]  If $p=2$ then $u$ is H\"older continuous in $\left\{u>0\right\}$ of order $\kappa(q)$, where $\kappa(q)$ is defined by
		\begin{equation}\label{order_hold_10}
			\kappa(q)=
			\begin{cases} 
				\tau, & \mbox{if } q\le 3, \\ 
				\tfrac{\tau}{q-2}, & \mbox{if }  q>3,
			\end{cases}
		\end{equation}
		for every $\tau\in (0,\min\left\{\alpha,1\right\})$.
	\end{itemize}
\end{lemma}

\begin{proof}
	Assume $u\in L^{s}(\R^N)\cap  L^2(\R^N)\cap L^q(\R^N)$, where $s\in \big(p, \frac{Np}{\alpha}\big)$.
	Note that $u\in L^2(\R^N)\cap L^q(\R^N)$ implies that $I_\alpha*u^p$ is almost everywhere finite on $\R^N$.
	Moreover, by the HLS inequality, if $s\in (p, Np/\alpha)$ then $I_{\alpha}\ast u^p\in L^{\tau}(\R^N)$ where
	\begin{equation}\label{HLS}
		\frac{1}{\tau}=\frac{p}{s}-\frac{\alpha}{N}>0.
	\end{equation}
	Then \eqref{eq11} implies that $(I_{\alpha}\ast u^p)u^{p-1}\in L^{\sigma}(\R^N)$ for
  $\sigma \geq1$ and $(I_{\alpha}\ast u^p)u^{p-1}\in \mathcal L^{1/\sigma}(\R^N)$ for
  $\sigma \in (0, 1)$.\footnote{We denote $\mathcal L^t(\R^N)=\{f:\R^n\to\R:\int|f|^t
    dx<\infty\}$, where $t\in(0,1)$. Note that $\mathcal L^t(\R^N)$ is no longer
  a normed space for $t \in (0,1)$ because the triangle inequality does not hold.} 
	\smallskip
	
	Next, we split the argument in three cases:
	\smallskip
	
	{\sc Case 1: $q>\frac{Np}{\alpha}$.}  In this case, $u\in L^{\bar{q}}(\R^N)$ for some  $\bar{q}\in \left(\frac{Np}{\alpha},q\right ]$ such that $\alpha-\frac{N}{
		\bar{q}}<1$.  Then $I_{\alpha}*u^p\in L^{\infty}(\bbrn)$ and is H\"older continuous of order $\alpha-\frac{N}{\bar{q}}$ (cf \cite[Theorem~2]{P55}).  	
	
	{\sc Case 2: $q=\frac{Np}{\alpha}$.} Since in this case there exists $\epsilon>0$ small such that  $u\in L^{p\left(\tfrac{N}{\alpha}-\epsilon\right)}(\R^N)$, from  \eqref{HLS} we get $(I_\alpha* u^p)u^{p-1}\in L^{\sigma}(\bbrn)$ where 
	$$\frac{1}{\sigma}=\frac{2p-1}{p\left(N/\alpha-\epsilon\right)}-\frac{\alpha}{N}.$$
	Thus, recalling  that
	$$
	u^{q-1}\le u+u^{q-1}=(I_{\alpha}\ast u^p)u^{p-1}\quad\text{a.e.~in $\R^N$,}
	$$
	$u\in L^{(q-1)\sigma}(\bbrn)$ and $(q-1)\sigma>\frac{Np}{\alpha}$ provided 
	$0<\epsilon<\frac{N}{\alpha}\left(1-\frac{2p-1}{p+q-1}\right)$. Thus,
	$u^p\in L^{
		\frac{N}{\alpha}-\epsilon}(\bbrn)\cap L^{\tfrac{(q-1)\sigma}{p}}(\bbrn)$, therefore
	$I_\alpha*u^p\in L^\infty(\R^N)$ and is H\"older continuous of order $\gamma$ for some $\gamma\in (0,1]$.
	
	{\sc Case 3: $q\in (\frac{2Np}{N+\alpha}, \frac{Np}{\alpha})$}. Let's set $s_0:=q>p$ and
	\begin{equation}\label{eq10}
		\frac{1}{s_{n+1}}:=\frac{1}{(q-1)\sigma_n}=\frac{2p-1}{(q-1)s_n}-\frac{\alpha}{N(q-1)}.
	\end{equation}
	Then $u^p\in L^{\frac{s_0}{p}}$ and $(I_{\alpha}\ast u^p)u^{p-1}\in L^{\sigma_0}(\R^N)$.  
	Hence we conclude that $u\in L^{(q-1)\sigma_0}(\bbrn)=L^{s_1}(\bbrn)$. Note that $q>\frac{2Np}{N+\alpha}$ implies $s_1>s_0=q$. In particular,
  if $s_n<\frac{Np}{\alpha}$, an induction argument yields $s_n<s_{n+1}$. This proves that $(s_n)$ is monotone increasing, 
  as long as $s_n$ is between $q$ and $Np/\alpha$.

	We claim that, after finite steps, there exists $n_0\in \bbn$ such that $s_{n_0}\geq \frac{Np}{\alpha}$ and $s_n<\frac{Np}{\alpha}$ for all $n<n_0$.   If not, we can obtain a sequence $(s_n)$ satisfying \eqref{eq10} and $q<s_n<\frac{Np}{\alpha}$ for all $n\in \bbn$. 
  By the monotonicity of this sequence, we also conclude that 	$s_n$ 
  converges to the unique fixed point $s_* = N(2p-q)/\alpha > q$, which contradicts the condition $q > 2Np/(N+\alpha)$.
	
 Then, if $s_{n_0}>\frac{Np}{
		\alpha}$ we conclude that $I_{\alpha}*u^p\in L^{\infty}(\R^N)$ and is H\"older continuous of order $\alpha-\frac{N}{s_{n_0}}$. If $s_{n_0}=\frac{Np}{
		\alpha}$ we can argue as in the previous case and we still obtain boundedness and H\"older regularity of $I_{\alpha}*u^p$.
	\smallskip
	
	Next, from $I_{\alpha}*u^p\in L^{\infty}(\R^N)$ and the relation
	\begin{equation}\label{eq41_0}
		u^{2-p}+u^{q-p}= I_\alpha*u^p\quad \text{a.e. in $\left\{u>0\right\}$},
	\end{equation}
	we conclude that $u\in L^\infty(\R^N)$. 
	Therefore, $u\in L^s(\bbrn)$ for all $s\geq 2$ from which  $I_\alpha*u^p$ is H\"older continuous of order $\tau$ for any $\tau\in (0, \min\{1, \alpha\})$.
	
	Furthermore, if $p< 2$,  since the function $f(t)=t^{2-p}+t^{q-p}$ has a differentiable inverse on $(0,\infty)$ and  $u\in L^{\infty}(\R^N)$,  it follows from \eqref{eq41_0} that $u$ has the same H\"older regularity as $I_\alpha*u^p$ in $\left\{u>0\right\}$.
	
	Similarly, if $p=2$, the function $f(t)=1+t^{q-2}$ has a differentiable inverse on $(0,+\infty)$ if $q\le 3$, and a locally H\"older inverse of order $\frac{1}{q-2}$ if $q>3$. Then again from the boundedness of $u$ and \eqref{eq41_0} we obtain \eqref{order_hold_10}. 
\end{proof}

\begin{corollary}\label{regu_p_big_12-comp}
	Assume that	$p\ge 2$ and $q>\frac{2Np}{N+\alpha}$.
	Let $u\in L^2(\R^N)\cap L^q(\R^N)$ be a nonnegative radial nonincreasing solution of \eqref{eqTF}. Then $u$ is compactly supported.
\end{corollary}

\begin{proof}
	Since $u$ is radially nonincreasing, by an abuse of notation we still denote $u(r)=u(|x|)$ where $r=|x|$.  Now, since $u(r)$ is a nonincreasing function, it can have at most a countable number points of discontinuity. Then, without loss of generality, if $r'$ is a discontinuity point, we define
	\begin{equation}\label{re_def}
		u(r'):=\lim_{r\to {r'}^{+}}u(r).
	\end{equation}
	Note that the above limit exists by monotonicity of $u(r)$ and, by doing this, we are only modifying $u$ on a set of measure zero.  In fact, 
	\begin{equation}\label{liminf_13}
		u(r')=\liminf_{r\to r'}u(r) ,
	\end{equation}
	which makes $u$ a lower semi-continuous function, and the set $\{u>0\}$ is open.
	
	Arguing by contradiction, we assume that $\left\{u>0\right\}=\R^N$.  Since $u$ is nonnegative and satisfies \eqref{eqTF}, we have 
	\begin{equation}\label{eq15}
		1\leq 1+u^{q-2}=(I_{\alpha}*u^p)u^{p-2}\quad \forall \ x \in \R^N.
	\end{equation}
	On the other hand, $I_{\alpha}*u^p$ vanishes at infinity by \eqref{l-potential-0}, and $u\in L^\infty(\R^N)$ by Lemma \ref{lem34}. Hence there exist $R>0$  such that $(I_{\alpha}*u^p)u^{p-2}<1$ in $ B_{{R}}^c$, a contradiction to \eqref{eq15}.
\end{proof}

Next we show that when $p>2$, nonnegative solutions of \eqref{eqTF} are discontinuous at the boundary of the support and H\"older continuous inside the support.

\begin{lemma}\label{regu_p_big_12}
	Assume that	$p> 2$ and $q>\frac{2Np}{N+\alpha}$.
	Let $u\in L^2(\R^N)\cap L^q(\R^N)$ be a nonnegative radial nonincreasing solution of \eqref{eqTF}. Then there exists 
	\begin{equation}\label{eq-lambda}
	\lambda\ge\lambda_*:=\left(\frac{p-2}{q-p}\right)^{\frac{1}{q-2}}
	\end{equation}
	such that $\left\{u>0\right\}=\left\{u>\lambda\right\}$ and $u$ is H\"older continuous of order $\kappa(p,q,\lambda)$ in $\left\{u>0\right\}$, where  
	\begin{equation}\label{order_hold}
		\kappa(p,q,\lambda)=
		\begin{cases} \tau, & \mbox{if } \lambda>\lambda_*, \\ 
			\tfrac{\tau}{2}, & \mbox{if }\lambda=\lambda_*,
		\end{cases}
	\end{equation}
	for every $\tau\in (0,\min\left\{\alpha,1\right\})$.
\end{lemma}

\begin{proof}
	Set $B_{R_*}:=\left\{u>0\right\}$, where $R_*<\infty$ in view of Corollary \ref{regu_p_big_12-comp}. Assume by contradiction that there exists a sequence $(r_n)\subset (0,R_*)$ such that  $r_n\rightarrow R_*$ and $u(r_n)\rightarrow 0$. Then, by \eqref{eq41_0},
	\begin{equation}\label{contrad_13}
		(I_{\alpha}*u^p)(r_n)= u(r_n)^{2-p}+u(r_n)^{q-p}\to \infty.
	\end{equation}
	However, from Lemma \ref{lem34}, the left hand side of \eqref{contrad_13} is bounded which leads to a contradiction. We have therefore proved that $u$ is far away from zero inside its support, or equivalently that there exists $\lambda>0$ such that $\left\{u>0\right\}=\left\{u>\lambda\right\}$.
	
	In what follows, we prove continuity of $u$ in $B_{R_*}$.
	First, we recall that  $I_{\alpha}*u^{p}$ is H\"older continuous by Lemma \ref{lem34}, and is radially nonincreasing, since $u$ is radially nonincreasing.
	Next,  we define the quantities
	\begin{equation}
		\lambda=\lim_{r\to R^{-}_*}u(r),\quad  \gamma=u(0).
	\end{equation}
	Note that $\lambda_{*}$, defined in \eqref{eq-lambda}, is the unique minimum of the function $f$ defined by 
	\begin{equation}\label{def_f}
		f(t)=t^{2-p}+t^{q-p}\qquad(t\in(0,+\infty)).
	\end{equation}
	To prove the continuity of $u$, as we will see shortly, it is enough to prove that 
	$\lambda\ge \lambda_*$. To this aim, we split the proof into two steps.
	\smallskip
	
	{\sc Step 1: $\gamma>\lambda_{*}$.}
	Assume by contradiction that $\gamma\le \lambda_*$. Since $u$ is radially nonincreasing, $u(r)\le \lambda_*$ for every $r\in (0,R_*)$. Furthermore, 
	since the function $f(t)=t^{2-p}+t^{q-p}$ is decreasing  in the interval $(0,\lambda_*]$,  we deduce that  $f(u(r))$ is non decreasing in $(0,R_*)$. Thus,
  from the equality $f(u(r))=(I_{\alpha}*u^p)(r)$ and  monotonicity of $I_{\alpha}*u^p$, and injectivity of $f$ (or strict monotonicity of $f$) the function $u$
  must be constant inside the support. Namely, $u(x)=\gamma \chi_{B_{R_*}}(x)$ and, from \eqref{eq41_0},
	\begin{equation}\label{eq_in_B}
    \gamma^{2-p}+\gamma^{q-p}=\gamma^{p} I_{\alpha}*\chi_{B_{R_*}}\quad \mbox{ in }\ B_{R_*}.
	\end{equation}
  In fact,  
  $I_{\alpha}*\chi_{B_{R_*}}$  can be written in terms of the Gauss Hypergeometric function as 
	\begin{equation}\label{riesz_13bis}
    (I_{\alpha}*\chi_{B_{R_*}})(x)=\frac{\Gamma((N-\alpha)/2)R_*^{\alpha}}{2^\alpha\Gamma(1+\alpha/2)\Gamma(N/2)}
    {}_{2}F_1\left(-\frac{\alpha}{2},\frac{N-\alpha}{2};\frac{N}{2};\frac{|x|^2}{R_*^2}\right),
	\end{equation}
  which is never a constant for $\alpha \in (0,N)$. We have therefore proved that $\gamma>\lambda_*.$

	\smallskip
	
	{\sc Step 2: $\lambda\ge \lambda_*$.} Assume that $\lambda<\lambda_*.$ First of all, we notice that $u$ can not be continuous in $(0,R_*)$. As a matter of fact, if $u$ is continuous, since by Step 1 we have that $\lambda_{*}\in (\lambda,\gamma) $,  the value $\lambda_*$ is achieved by $u$. Namely,  there exists $\bar{r}\in (0,R_*)$ such that $u(\bar{r})=\lambda_{*}$. Arguing as before, since $u(r)$ is nonincreasing, $f$ is decreasing in $[\lambda,\lambda_{*}]$, and $(I_{\alpha}*u^p)(r)$ is nonincreasing.  We infer that $u(r)$ is constant for every $r\in [\bar{r},R_*)$. However, this implies that 
	$$\lambda=\lim_{r\to R^{-}_*}u(r)=u(\bar{r})=\lambda_{*},$$ which is a contradiction.
	
	Next, we show that  if  $r'$ is a discontinuity point  of $u(r)$, we must have that
	$u(r')\in [\lambda,\lambda_{*}].$
	Indeed, by \eqref{liminf_13}, if  $u(r')=L\in (\lambda_{*},\gamma]$ then for every $\epsilon>0$ sufficiently small there exists $\delta>0$ such that $u(r)\ge L-\epsilon$ for every $r\in (r'-\delta,r'+\delta)$. In particular, if we choose $\epsilon$ such that $L-\epsilon>\lambda_{*}$, we deduce that  $u((r'-\delta,r'+\delta))\subset (\lambda_{*}, \gamma]$. But in this interval the function $f$ is invertible with continuous inverse. Then, 
	$$u(r)=f^{-1}\big((I_{\alpha}*u^p)(r)\big)\quad  \forall r\in (r'-\delta,r'+\delta),$$
	which in particular implies continuity of $u$ at $r'$ and this is a contradiction.
	
	Next, in view of  monotonicity of $u(r)$, we conclude that $u([r', R_*))\subset  [\lambda,\lambda_{*}]$. Then, since in $[\lambda,\lambda_{*}]$ the function $f$ is decreasing, again monotonicity of $u$ implies that $u(r)=\lambda$ for every $r\in [r', R_*]$. Finally, it remains to prove that this is not possible and this will imply that $\lambda\ge \lambda_{*}$.
	
	Since we have assumed that  $u(r)$ is nonincreasing and  constant in $[r',R_*]$, there exists $\lambda_1>\lambda$ such that 
	\begin{equation}\label{replace_15}
		u^p=\lambda^{p}\chi_{B_{R_*}}+ \phi+ (\lambda^p_1-\lambda^p)\chi_{B_{r'}},
	\end{equation}
	where $\phi$ is a  radially nonincreasing function such that
	$\phi(r)=0$ if $r\ge r'$. 
	Thus, by combining \eqref{eq41_0} with \eqref{replace_15} we obtain the equality
	\begin{equation}\label{riesz_14}
		\lambda^{2-p}+\lambda^{q-p}=\lambda^{p}(I_{\alpha}*\chi_{B_{R_*}})(r)+(\lambda^p_1-\lambda^p)(I_{\alpha}*\chi_{B_{r'}})(r)+(I_{\alpha}*\phi)(r)\quad \forall r\in  (r', R_*).
	\end{equation}
	However, again by \eqref{riesz_13bis}, the right hand side of \eqref{riesz_14} is decreasing in $(R_*-\epsilon,R_*)$ for some  $\epsilon>0$  small enough and this contradicts \eqref{riesz_14}.
	\smallskip
	
	We have therefore proved that  $\lambda\ge \lambda_*$. Then on the set $[\lambda_*,\infty)$ the function $f$ has an inverse $f^{-1}:[\lambda_*,\infty)\to {[\Lambda_*,\infty)}$,  where we denote $\Lambda_*:=f(\lambda_*)$. We conclude that 
	\begin{equation}\label{eq-f-inv}
		{u=f^{-1}(I_{\alpha}*u^p)\quad\text{in $B_{R_*}$}.}
	\end{equation}
	To prove that the desired H\"older exponent given by \eqref{order_hold} we consider two different cases.
	\smallskip
	
	\textbf{Case a)}: $\lambda>\lambda_*$. 
	In this case $f$ is a Lipschitz function with Lipschitz inverse in the set $$u(B_{R_*}):=\left\{u(x):x\in B_{R_*}\right\}$$ 
	and by Lemma \ref{lem34} we have 
	$u=f^{-1}(I_{\alpha}*u^{p})\in C^{0,\tau}(B_{R_*})$ for every $\tau\in (0,\min\left\{1,\alpha\right\})$. 
	\smallskip
	
	\indent	\textbf{Case b)}: $\lambda=\lambda_{*}$. In this case, let's notice that
	$f''(\lambda_{*})=p(q-p)^2\lambda^{q-p-2}_{*}>0$,
	which means that if $\epsilon>0$ is small enough,  the following expansion holds
	\begin{equation}\label{taylor_15}
		f(t)=f(\lambda_*)+\frac12 f''(\lambda_*)(t-\lambda_*)^{2}+o\big((t-\lambda_*)^{2}\big) \quad \forall t\in (\lambda_*-\epsilon,\lambda_* +\epsilon).
	\end{equation}
	Let $f^{-1}$ be the inverse of $f$ on $[\lambda_{*},\infty)$.  
	Then, if for  $s\ge\Lambda_*$ we set $t:=f^{-1}(s)$, by \eqref{taylor_15} we obtain
	\begin{equation*}
		\lim_{s\to \Lambda_*^{+}}\frac{|f^{-1}(\Lambda_*)-f^{-1}(s)|}{|\Lambda_*-s|^{\frac{1}{2}}}
		= \lim_{t\to \lambda^{+}_*}\frac{|\lambda_*-t|}{\left|\frac12 f''(\lambda_*)(t-\lambda_*)^{2}+o(t-\lambda_*)^{2}\right|^{\frac{1}2}}
		=\sqrt{\frac{2}{f''(\lambda_*)}},
	\end{equation*}
	which proves that $f^{-1}$ is H\"older continuous of order $1/2$. Then using
	Lemma \ref{lem34} we obtain
	$u=f^{-1}(I_{\alpha}*u^{p})\in C^{0,\frac{\tau}{2}}(B_{R_*})$ for every $\tau\in (0,\min\left\{1,\alpha\right\})$.
\end{proof}

  \begin{remark}\label{r-lambda}
    Although the above discussion about the regularity depends on the jump $\lambda$ of the solution near the boundary of the support, compared with 
    $\lambda_*=\left(\frac{p-2}{q-p}\right)^{1/(q-2)}$. Numerical experiments suggests that $\lambda$ is always strictly larger than $\lambda_*$, as
    show in Figure~\ref{fig:regcomp}, although the difference becomes smaller for smaller $\alpha$.  
    \begin{figure}[htp]
      \label{fig:regcomp}
      \begin{center}
        \includegraphics[totalheight=0.26\textheight]{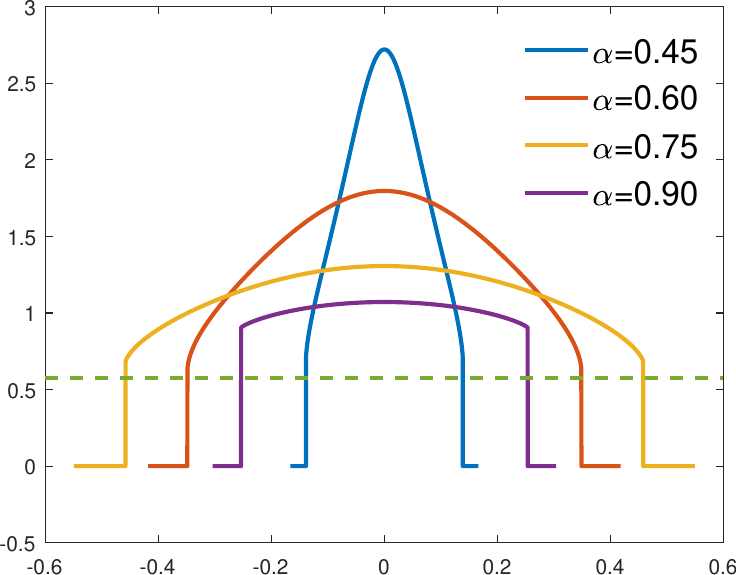}
      \end{center}
      \caption{The jump near the boundary compared with $\lambda^*$(the dashed line) in one dimension with $p=2.5, q=4$ and different $\alpha$.}
    \end{figure}
  \end{remark}

Finally, similarly to \cite{Carrillo-CalcVar}*{Theorem $10$} we show that nonnegative solutions of \eqref{eqTF} are smooth inside the support.

\begin{lemma}\label{smooth_support}
	Assume	that $p\ge 2$ and $q>\frac{2Np}{N+\alpha}$. Let $u\in L^2\cap L^q(\R^N)$ be a nonnegative radial nonincreasing solution of \eqref{eqTF}.  Then $u\in C^{\infty}$ inside its support.
\end{lemma} 

\begin{proof}
	Assume first that $0<\alpha<2$.
	As in Lemma \ref{regu_p_big_12}, denote by $B_{R_*}=\left\{u>0\right\}$. Let $x\in B_{R_*}$ and $B$ be a ball centered at $x$, such that $\overline{B}\subset B_{R_*}$. By Lemma \ref{lem34}, we know that $I_{\alpha}*u^p\in C^{0,\tau}(\R^N)$ for every  $\tau\in (0, \min\{\alpha, 1\})$. Then, as in the proof of Lemma \ref{regu_p_big_12},  
	\begin{equation}\label{u_inverse}
		u=f^{-1}(I_{\alpha}*u^p),
	\end{equation}
	where $f^{-1}$ is the inverse of $f$ on $[\lambda_*,\infty)$. In particular, since in $\overline{B}$ the function $u$ is away from $\lambda$ if $p>2$ (respectively from $0$ if $p=2$), then $u$ (and so $u^p$) has the regularity of $I_{\alpha}*u^p$. Namely,  $u^p\in  C^{0,\tau}(\overline{B})$ for every  $\tau\in (0, \min\{\alpha, 1\})$. Then, \cite{Ros}*{Theorem $1.1$, Corollary $3.5$} yields $I_{\alpha}*u^p\in C^{0,\tau+\alpha}(\overline{B}_{1/2})$ for every  $\tau\in (0, \min\{\alpha, 1\})$, provided that $\tau+\alpha$ is not an integer. Here by $B_{1/2}$ we denoted the ball centered at $x$ with half of the radius of $B$, and as usual, $C^{0,\tau+\alpha}=C^{\gamma',\gamma''}$ with $\tau+\alpha=\gamma'+\gamma''$ and  $\gamma'$ is the largest integer less or equal than $\tau+\alpha$. Hence, again by \eqref{u_inverse},  we conclude that $u$ has the  regularity of $I_{\alpha}*u^p$. By iterating the above argument, for every $k\in \N$ we can find $j\in \N$ such that $\tau+j \alpha$ is non integer and bigger than $k$. This proves that $u\in C^{k}(\overline{B}_{1/2^{j}})$.
	Since $x$ was arbitrary, this implies that $u$ is smooth inside its support.
	
	If $2\le \alpha<N$ we can argue again similarly to the proof of \cite{Carrillo-CalcVar}*{Theorem  10}.
\end{proof}

\subsection{Support estimates for $p> 2$}

The following two statements follow from estimates in Proposition \ref{limit_alpha}.

\begin{corollary}\label{cor-supp-0}
	Let $p>2$ and $q>2p$. Let $u\in L^2(\R^N)\cap L^q(\R^N)$ be a nonnegative radial nonincreasing solution of \eqref{eqTF} and  $R_*$ be the radius of the support of $u$. Then 
	$R_*\to +\infty$ as $\alpha\to 0$. 
\end{corollary}
	\begin{proof}
    By combining the Nehari identity~\eqref{Nehari} and 
    Poho\v zaev  identity $\mathcal{P}(u)=0$, we get
\begin{equation}\label{L_2_less_L_q}
	\|u\|^{2}_2=\frac{(N+\alpha)q-2pN}{q(Np-N-\alpha)} \|u\|^q_q.
\end{equation}
As a result, the minimal energy can also be represented using $\|u\|_q$, that is,
\begin{equation}\label{sigma_19}
	\sigma_*=\inf_{v\in \mathscr{P}}E(v)=E(u)=\frac{\alpha(q-2)}{2(Np-N-\alpha)q}\|u\|^{q}_q.
\end{equation}
In what follows we will write $\sigma_*=\sigma_*(\alpha)$ stressing the dependence on $\alpha$.
Moreover, by the monotonicity of $u(|x|)$ in $|x|$,
\begin{equation}\label{2_sided}
	\|u\|^{q}_q< u^{q}(0)|B_{R_*}|.
\end{equation}
Evaluating the governing equation \eqref{eqTF} at $0$, we estimate
\begin{equation}\label{gamma_0}
  u^{q-p}(0)\le u^{2-p}(0)+u^{q-p}(0)=(I_{\alpha}*u^p)(0),
\end{equation}
from which 
\begin{equation*}
	u(0)\le \big((I_{\alpha}*u^p)(0)\big)^{\frac{1}{q-p}}
  =\left(A_\alpha \int_{B_{R_*}}u^p(y)|y|^{\alpha-N}dy\right)^{\frac{1}{q-p}}\le u^{\frac{p}{q-p}}(0)
  \left(A_\alpha \omega_{N} \tfrac{R^{\alpha}_*}{\alpha}\right)^{\frac{1}{q-p}},
\end{equation*}
where $\omega_N$ is the surface area of the unit sphere in $\R^N$. 
The previous inequality leads to the bound 
\begin{equation*}
	u(0)\le \left(A_\alpha \omega_{N} \tfrac{R^{\alpha}_*}{\alpha}\right)^{\frac{1}{q-2p}}.
\end{equation*}
This estimate, when combined with  \eqref{2_sided}, turns the relation \eqref{sigma_19}
into 
\begin{equation}\label{low_radius}
	|B_{R_*}|R^{\frac{\alpha q}{q-2p}}_*\ge  \frac{2q(Np-N-\alpha)}{\alpha(q-2)} \left(\frac{A_\alpha \omega_{N}}{\alpha}\right)^{-\tfrac{q}{q-2p}} \sigma_*(\alpha)
\end{equation}
Next, if we consider $\theta_*=\theta_*(\alpha)$ defined as in \eqref{theta-star}, we have that 
\begin{equation}
	\lim_{\alpha\to 0} \theta_*(\alpha)=\left(\frac{q(p-1)}{q-2p}\right)^{\frac{q-2p}{2(p-1)+q-2p}} \left(\frac{q-2p}{2q(p-1)}+\frac{1}{q}\right)=\overline{\theta_*}.
\end{equation}
Furthermore,  under our assumptions on $p$ and $q$, we have 
$ \lim_{\alpha\to 0 } 2p \theta_*(\alpha)=2p \overline{\theta_*}>1.$
Moreover, from \eqref{low_radius} we deduce that
\begin{equation}
	R_*\ge \omega^{-\frac{2(q-p)}{q(N+\alpha)-2Np}}_{N} \left(\frac{2q(Np-N-\alpha)}{(q-2)}\right)^{\frac{q-2p}{q(N+\alpha)-2Np}} \left(\frac{A_\alpha}{\alpha}\right)^{-\tfrac{q}{q(N+\alpha)-2Np}} \left(\frac{\sigma_*(\alpha)}{\alpha}\right)^{\frac{q-2p}{q(N+\alpha)-2Np}}.
\end{equation}
Since $\lim_{\alpha\to 0}\alpha ^{-1} A_{\alpha}=\omega^{-1}_{N}$, it remains to prove that $\lim_{\alpha\to 0} \alpha^{-1}\sigma_*(\alpha)=\infty$. To do so, from   \eqref{rel_minimi}, \eqref{low_radius} and Proposition \ref{limit_alpha} we infer 
\begin{equation}
	\lim_{\alpha\to 0}\frac{\sigma_*(\alpha)}{\alpha}= \frac{\overline{\theta_*}}{N} \lim_{\alpha\to 0} \left(\frac{N}{N+\alpha}\frac{2p \theta_*(\alpha)}{\mathscr{C}_{N,\alpha,p,q}}\right)^{\frac{N}{\alpha}}=\infty, 
\end{equation}
where in the last equality we used that 
$$ \lim_{\alpha\to 0} \left(\frac{N}{N+\alpha}\frac{2p \theta_*(\alpha)}{\mathscr{C}_{N,\alpha,p,q}}\right)=2p \overline{\theta_*}>1.$$
This concludes the proof.
\end{proof}
	
\begin{corollary}\label{cor-supp-N}
Let $p>2$. Let $u\in L^2(\R^N)\cap L^q(\R^N)$ be a nonnegative radial nonincreasing solution of \eqref{eqTF} and $R_*$ be the radius of the support of $u$. Then, $R_*\to 0$ as $\alpha\to N$.
\end{corollary}
\begin{proof}
By combining  \eqref{sigma_19}, Lemma \ref{regu_p_big_12} with the monotonicity of $u$ we obtain that 
\begin{equation}\label{r_N}
  |B_{R_*}|\le  \left(\frac{q-p}{p-2}\right)^{\frac{q}{q-2}}\,\frac{2 (Np-N-\alpha)q}{\alpha(q-2)} \sigma_*(\alpha)
\end{equation}
Furthermore, by combining Proposition \ref{limit_alpha} with \eqref{rel_minimi} we conclude that $\lim_{\alpha\to N}\sigma_*(\alpha)=0$. Thus the conclusion follows in view of \eqref{r_N}.
\end{proof}

\subsection{Gradient estimates for $p<2$}
In the rest of the section we consider the case $\frac{N+\alpha}{N}<p< 2$. Recall that in this case nonnegative radial nonincreasing solutions $u\in
L^2(\R^N)\cap L^q(\R^N)$ of \eqref{eqTF} are supported on $\R^N$. Our aim is to show that $\nabla u\in L^2(\R^N;\R^N)$. 

Note that for $\alpha>1$ the gradient $\nabla I_{\alpha}*u^p$ is well defined, while for $0<\alpha\leq 1$ it becomes a singular integral and is defined via the Cauchy principal value, namely 
\begin{equation}\label{grad_17}
  \nabla (I_{\alpha}*u^p)=\left\{\aligned &(\nabla I_{\alpha})*u^p,\quad
	&&\alpha>1,\\
	&\int_{\bbrn}\nabla|x-y|^{\alpha-N}(|u(y)|^p-|u(x)|^p)dy, \quad &&0<\alpha\leq 1, \\
	\endaligned\right.
\end{equation}
cf. \cite{Carrillo-CalcVar}*{eq. (1.2)}.
Recall the following result from \cite{Carrillo-CalcVar}*{Lemma 1}.

\begin{lemma}\label{lem621}
 Assume that $u\ge 0$ and $u^p\in L^1(\bbrn)\cap L^{\infty}(\bbrn)$.  Then 
	\begin{itemize}
		\item [$(i)$] If $0<\alpha<N$, then $I_{\alpha}*u^p\in L^{\infty}(\bbrn)$.\smallskip
		
		\item [$(ii)$] If $0<\alpha\leq 1$ and $u^p\in C^{0, \gamma}(\bbrn)$ with $\gamma\in (1-\alpha, 1)$, or if $1<\alpha<N$ then $\nabla (I_{\alpha}*u^p)\in L^{\infty}(\bbrn)$, i.e., $I_{\alpha}*u^p\in W^{1, \infty}(\bbrn)$.
	\end{itemize}
\end{lemma}

Using the estimates of Lemma \ref{lem621}, we first show that positive solutions of \eqref{eqTF} are globally Lipschitz.

\begin{lemma}\label{lem622}
	Assume that	$\frac{N+\alpha}{N}<p< 2$ and $q>\frac{2Np}{N+\alpha}$. 
  Let $u\in L^2(\R^N)\cap L^q(\R^N)$ be a positive radial nonincreasing solution of \eqref{eqTF}.
  Then $I_{\alpha}*u^p\in W^{1, \infty}(\bbrn)$ and $u\in W^{1, \infty}(\bbrn)$.
\end{lemma}
\begin{proof}
	By Lemma \ref{lem51} and Lemma \ref{lem34}, we know that \text{Supp}$(u)=\R^N$ and $u^p\in L^{\infty}(\bbrn)$.
	
	If $1<\alpha<N$ we can apply Lemma \ref{lem621} to conclude that $I_{\alpha}*u^p\in W^{1, \infty}(\bbrn)$.
	Next, since $\text{Supp}(u)=\R^N$, $u$ satisfies the equivalent governing equation
	\begin{equation}\label{eq_R}
		u^{2-p}+u^{q-p}=I_{\alpha}*u^p\quad \text{in}\ \R^N.
	\end{equation}
	Finally, as we have already noticed in the proof of Lemma \ref{lem34}, the function $f(t)=t^{2-p}+t^{q-p}$ has a differentiable inverse on $(0,+\infty)$ under our assumptions on $p$ and $q$. From  \eqref{eq_R} we conclude that $u\in W^{1,\infty}(\R^N)$.
	
	If $0<\alpha\leq 1$ then Lemma \ref{lem34} yields that $I_{\alpha}*u^p\in C^{0,\tau}(\R^N)$ for every $\tau\in (0,\alpha)$. Thus, again from \eqref{eq_R}, the
  differentiability of the inverse $f^{-1}$ on $(0,+\infty)$ and the boundedness of $u$, we infer  $u\in C^{0, \tau}(\bbrn)$ for every $\tau\in (0,\alpha)$. 
	In particular, if $1/2<\alpha\leq 1$ we can ensure that $\tau>1-\alpha$, and  hence $I_\alpha*u^p\in W^{1,\infty}(\R^N)$ by Lemma \ref{lem621}. Then, arguing as before, $u\in W^{1,\infty}(\R^N)$.
	For $0<\alpha< 1/2$, on the other hand, we need to use bootstrapping argument. Let us fix $n\in \bbn$, $n\geq 2$ such that $\frac{1}{n+1}<\alpha<\frac{1}{n}$
  and let us choose $\tau>0$ small enough such that $\tau+n \alpha<1$ (note that this is possible because of the definition of $n$). Then, we define
  $\tau_n:=\tau+(n-1)\alpha$.  Then, by Eq. \eqref{eq_R} together with the locally Lipschitz continuity of the inverse of $f$,  we can  apply
  \cite{Silvestre}*{Proposition $2.8$} $n$--times to conclude that $I_{\alpha}*u^p\in C^{0, \tau_{n+1}}(\bbrn)$. By our choice of $n$, we have the two sided
  inequality $1-\alpha<\tau_{n+1}<1$.  Hence, by Eq. \eqref{eq_R} again, we deduce that $u$ (and in particular $u^p$) belongs to $C^{0,\tau_{n+1}}(\R^N)$. To conclude, by Lemma \ref{lem34} we conclude that $I_{\alpha}*u^p\in W^{1,\infty}(\R^N)$, and that $u$ has the same regularity.
	Finally, if $\alpha=\frac{1}{2}$ it is sufficient to start the the above iterations with $\tau-\epsilon$, for some $\epsilon>0$ small enough such that $1-\alpha<\tau_{n+1}-\epsilon<1$.
\end{proof}

Next we show that positive solutions are actually arbitrarily smooth.

\begin{lemma}\label{full_support_smooth}
	Assume that	$\frac{N+\alpha}{N}<p< 2$ and $q>\frac{2Np}{N+\alpha}$. 
  Let $u\in L^2(\R^N)\cap L^q(\R^N)$ be a positive radial nonincreasing solution of \eqref{eqTF}. Then, $u\in C^{\infty}(\R^N)$.
\end{lemma}

\begin{proof}
  This follows from Eq. \eqref{eq_R} as in the proof of Lemma \ref{smooth_support}.
\end{proof}

Next we establish a gradient estimate on the nonnegative solutions of \eqref{eqTF}.

\begin{lemma}\label{grad_fast}
	Assume that	$\frac{N+\alpha}{N}<p< 2$ and $q>\frac{2Np}{N+\alpha}$. 
  Let $u\in L^2(\R^N)\cap L^q(\R^N)$ be a nonnegative radial nonincreasing solution of \eqref{eqTF}.
	Then, $\nabla u\in L^{1}(\R^N)$. In particular, $u\in H^{1}(\R^N)$.
\end{lemma}

\begin{proof}
  Assume first that $1<\alpha<N$. From the expression in Eq. \eqref{grad_17}, we deduce that 
	\begin{multline}\label{fst_17}
			|\nabla (I_{\alpha}*u^p)(x)|\le \int_{\R^N}\big|\nabla |x-y|^{\alpha-N}\big| u^p(y) dy \cr
      \le (N-\alpha)A_{\alpha}\int_{\R^N}\frac{u^p(y)}{|x-y|^{N-\alpha+1}}dy
			 =\frac{(N-\alpha)A_\alpha \|u\|^p_{p}}{|x|^{N-\alpha+1}}+o\left(\frac{1}{|x|^{N-\alpha+1}}\right),
	\end{multline}
	for $|x|$ sufficiently large. 
	Note also that the inverse $f^{-1}$ is differentiable on $(0,+\infty)$ and 
	\begin{equation}\label{asy_0}
		\big(f^{-1}\big)^\prime(t)=t^{\frac{p-1}{2-p}}+o(t^{\frac{p-1}{2-p}})\quad  \text{as}\ t\to 0^{+}.
	\end{equation}
	Hence, by using the chain rule in \eqref{eq_R}, Lemma \ref{lemm_11}, \eqref{fst_17} and \eqref{asy_0}, we infer 
	\begin{equation}\label{grad_18}
		|\nabla u (x)|=\big|\big(f^{-1}\big)^\prime\big((I_\alpha*u^p)(x)\big) \nabla (I_\alpha*u^p)(x)\big|\lesssim \frac{1}{|x|^{\frac{N-\alpha}{2-p}+1}}\quad \text{as}\ |x|\to +\infty.
	\end{equation}
	Combining Lemma \ref{lem622} with \eqref{grad_18} yields  $\nabla u\in L^1(\R^N)\cap L^{\infty}(\R^N)$ which concludes the proof.
	
	Assume now that $\alpha\in (0,1]$. From \eqref{grad_17}, arguing as in \cite{Carrillo-NA}*{Lemma $2.2$} we have 
	\begin{equation}
		\begin{split}
			\nabla (I_{\alpha}*u^p) & \le (N-\alpha)A_{\alpha}\left( \int_{|x-y|\le 1}\frac{|u^p(y)-u^p(x)|}{|x-y|^{N-\alpha+1}} +\int_{|x-y|>1}\frac{u^p(y)}{
				|x-y|^{N-\alpha+1}}\right)\\
			& =I_1+I_2.
		\end{split}
	\end{equation}
  First we note that, since $\alpha\in (0,1],$ for every $\ve\in (0,N) $, $I_2$ can be estimated as 
	\begin{multline}\label{estimate_0}
			\int_{|x-y|>1}\frac{u^p(y)}{
				|x-y|^{N-\alpha+1}}\le \int_{|x-y|>1}\frac{u^p(y)}{
				|x-y|^{N-\ve}} \cr \le \int_{\R^N}\frac{u^p(y)}{|x-y|^{N-\ve}}dy
			 = \frac{\|u\|^p_{p}}{|x|^{N-\ve}}+o\left(\frac{1}{|x|^{N-\ve}}\right),
	\end{multline}
	where for the last equality we used the decay estimate \eqref{l-potential-0} on $u$ established in Lemma \ref{lemm_11}.
	
	Let us fix $0<\bar{\ve}<\frac{(N-\alpha)(p-1)}{2-p}$.
  Since $\nabla (I_{\alpha}*u^p)\in L^{\infty}(\bbrn)$, by applying the gradient operator to both sides of Eq. \eqref{eq_R} we deduce that 
	\begin{equation}\label{step_1}
		|\nabla u (x)|\lesssim \big|\big(f^{-1}\big)^\prime((I_\alpha*u^p)(x))\big|\lesssim \frac{1}{|x|^{\frac{(N-\alpha)(p-1)}{2-p}}}\quad \text{as}\ |x|\to +\infty,
	\end{equation}
  where for the last inequality we used Corollary \ref{cor_12} and Eq. \eqref{asy_0}. 

  Next, we estimate $I_1$. By combining Eq. \eqref{step_1} with Lemma \ref{lemm_11} we conclude that 
	\begin{multline}\label{estimate_1}
			I_1\lesssim \|\nabla u^p\|_{L^{\infty}(\overline{B_1(x)})}\int_{|x-y|\le1}\frac{dy}{|x-y|^{N-\alpha}}
      \cr \lesssim \|u^{p-1} \nabla u\|_{L^{\infty}(\overline{B_1(x)})}
			\lesssim \frac{1}{|x|^{\frac{2(N-\alpha)(p-1)}{2-p}}}\quad \text{as}\ |x|\to +\infty.
	\end{multline}
  Then, if $\frac{2(N-\alpha)(p-1)}{2-p}\le N-\bar{\ve} $, combining together Eq. \eqref{estimate_0} with Eq. \eqref{estimate_1} yields 
	\begin{equation}\label{induction_1}
		\nabla (I_{\alpha}*u^p)\lesssim \frac{1}{|x|^{\frac{2(N-\alpha)(p-1)}{2-p}}}+\frac{1}{|x|^{N-\bar{\ve}}}\quad \text{as}\ |x|\to +\infty.
	\end{equation}
  On the other hand, if $\frac{2(N-\alpha)(p-1)}{2-p}> N-\bar{\ve} $, by the same argument it follows that 
	$$	\nabla (I_{\alpha}*u^p)\lesssim \frac{1}{|x|^{N-\ve}}\quad \text{as}\ |x|\to +\infty $$
	which in turn implies that
	$$|\nabla u(x)|\lesssim \frac{1}{|x|^{\frac{(N-\alpha)(p-1)}{2-p}+N-\bar{\ve}}}\quad \text{as}\ |x|\to +\infty .$$ Then, by the choice of $\bar{\ve}$ we
  conclude that $\nabla u\in  L^{1}(\R^N)$. However, if  Eq.~\eqref{induction_1} holds, we can improve the inequality \eqref{step_1} to 
	\begin{equation}\label{step_2}
		|\nabla u(x)|\lesssim \frac{1}{|x|^{\frac{3(N-\alpha)(p-1)}{2-p}}}\quad \text{as}\ |x|\to +\infty.
	\end{equation}
	Then, we can iterate this argument, until we find the first  positive integer $k$ such that 
	$$\frac{2k(N-\alpha)(p-1)}{2-p}> N-\bar{\ve} .$$
	In this way we obtain that 
	$$|\nabla u(x)|\lesssim \frac{1}{|x|^{\frac{(N-\alpha)(p-1)}{2-p}+N-\bar{\ve}}}\quad \text{as}\ |x|\to +\infty,$$
	which again implies $\nabla u\in L^{1}(\R^N)$ by the choice of $\bar{\ve}$.
\end{proof}


\section{Limit profiles for the Choquard equation}\label{sec-limit}

Throughout this section we assume that $\frac{N+\alpha}{N}<p<\frac{N+\alpha}{N-2}$ and $q>\frac{2Np}{N+\alpha}$.
As already highlighted in the Introduction, the rescaling
$	u(x)=\eps^{-\frac{1}{q-2}}w\big(\eps^{-\frac{2p-q}{\alpha(q-2)}}x\big)$
converts the Choquard problem \eqref{eqPeps} into the equation
\begin{equation}\label{eqPeps-TF+}
	-\eps^{\nu}\Delta u+u+|u|^{q-2}u=(I_{\alpha}*|u|^p)|u|^{p-2}u\quad\text{in $\R^N$},
\end{equation}
where we denoted $\nu:=\frac{2(2p+\alpha)-q(2+\alpha)}{\alpha(q-2)}.$ 
Notice that:
\begin{itemize}
	\item[$(i)$]$\nu> 0$ iff  $q< 2 \frac{2p+\alpha}{2+\alpha}$,
	\smallskip	
	\item[$(ii)$] $\nu<0$ iff $q> 2\frac{2p+\alpha}{2+\alpha}$.
\end{itemize}
The energy that corresponds to the rescaled equation \eqref{eqPeps-TF+} is given by
\begin{equation}\label{eq27+}
	\mathcal{J}_{\eps}(u)=\frac12 \eps^\nu\int_{\R^N}|\nabla u|^2dx+\frac12\int_{\R^N}|u|^2dx
  +\frac{1}{q}\int_{\R^N}|u|^qdx
  -\frac{1}{2p}\D_\alpha(|u|^p, |u|^p) ,
\end{equation}
and its Poho\v zaev functional is defined by 
\begin{multline}
	\mathcal{P}_\ve(u)=\frac{N-2}{2}\eps^\nu\int_{\R^N}|\nabla u|^2dx+\frac{N}{2}\int_{\R^N}|u|^2dx
  +\frac{N}{q}\int_{\R^N}|u|^qdx
  -\frac{N+\alpha}{2p}\D_\alpha(|u|^p, |u|^p).
\end{multline}
We note that
\begin{equation}\label{eq02+}
	\mathcal{J}_{\eps}(u)=\eps^{ \frac{q(N+\alpha)-2Np}{\alpha(q-2)}}\mathcal{I}_{\eps}(w),
\end{equation}
where $\mathcal{I}_{\eps}(w)$ is the Choquard energy defined in \eqref{eq02}.
Following \cite{ZLVMeps}, we consider the rescaled minimization problem
\begin{equation}\label{eq-sigma-eps}
	\sigma_\ve:=\inf_{u\in \mathscr{P}_\ve} \mathcal{J}_{\eps}(u),
\end{equation}
where Poho\v zaev manifold $\mathscr{P}_\ve$ is defined as 
$\mathscr{P}_\ve=\left\{u \in H^1(\R^N)\cap L^q(\R^N), u\neq 0: \mathcal{P}_\ve(u)=0 \right\}$.

Given $\eps>0$, let $w_\eps \in H^1(\R^N)\cap L^1(\R^N)\cap C^2(\R^N)$ be a positive, radial, monotonically decreasing ground state solution of
\eqref{eqPeps} (see Theorem \ref{thm01eps}). Define
\begin{equation}\label{TF-scale-eps}
	u_\eps(x):=\eps^{-\frac{1}{q-2}}w_\eps\big(\eps^{-\frac{2p-q}{\alpha(q-2)}}x\big).
\end{equation}
Then $u_\eps \in \mathscr{P}_\eps$ and	$\mathcal{J}_{\eps}(u_\eps)=\sigma_\eps$,
that is $u_\eps$ is the minimizer of \eqref{eq-sigma-eps} and a positive ground state solution of the rescaled equation \eqref{eqPeps-TF+}, see \cite{ZLVMeps}. 

In this section we shall prove Theorem \ref{t-TF-0}, which states that $u_\eps$ converges as $\eps\to +\infty$ and $\nu<0$ (respectively as $\eps\to 0$ and
$\nu>0$)  to a nonnegative radial nonincreasing ground state solution $u_*\in L^2(\R^N)\cap L^q(\R^N)$ of the Thomas--Fermi equation \eqref{eqTF}, constructed in Lemma \ref{rescaling_1} from a maximizer in Theorem \ref{t-GN}. Recall that $E(u_*)=\sigma_*$, where 
\begin{equation}\label{rel_7+}
	\sigma_*=\inf_{u\in \mathscr{P}}E(u)=\alpha (2Np)^{\frac{N}{\alpha}}\left(\frac{\theta_*}{N+\alpha}\right)^{\frac{N+\alpha}{\alpha}}\C^{-\frac{N}{\alpha}}_{N,\alpha,p,q},
\end{equation}
as described in \eqref{riesz_13bis}. The essential step in our proof of convergence is to show that $\sigma_{\ve}\to \sigma_*$.
\smallskip

In what follows we shall only consider the case $\eps\to +\infty$ and $\nu<0$, i.e., $q> 2\frac{2p+\alpha}{2+\alpha}$. The arguments in the case $\eps\to 0$ and $\nu>0$ are very similar.

\smallskip

First we study the easier case when the ground state solution $u_*\in D^{1}(\R^N)$, where for $N\ge 3$ we denote 
$D^{1}(\R^N)=\big\{u\in L^\frac{2N}{N-2}(\R^N)\,:\,\|\nabla u\|_2<\infty\big\}$.
In particular, $u_*\in D^{1}(\R^N)$ if $p<2$, as proved in Lemma \ref{grad_fast}. 

\begin{lemma}\label{lemma_6} 
	Assume $\frac{N+\alpha}{N}<p<\frac{N+\alpha}{N-2}$ and  $q> 2\frac{2p+\alpha}{2+\alpha}$.  
	If $u_*\in D^{1}(\R^N)$ then 
  (1) $\sigma_\varepsilon > \sigma_*$ and (2) $\sigma_\varepsilon\to\sigma_*$ as $\varepsilon\rightarrow +\infty$.
\end{lemma}
\begin{proof}
  Let $u_\varepsilon\in \mathscr{P}_{\varepsilon}$ be the minimizer of the problem $\displaystyle\inf_{u\in \mathscr{P}_\ve} \mathcal{J}_{\eps}(u)$  
  such that $\mathcal{J}(u_\ve)=\sigma_\ve$. Then, 
	\begin{equation*}
		\mathcal{P}(u_{\varepsilon})=\mathcal{P}_{\varepsilon}(u_{\varepsilon})-\left(\frac{N-2}{2}\right){\varepsilon}^{\nu}\|\nabla u_{\varepsilon}\|^{2}_2<0.
	\end{equation*}
	Let $w_{\varepsilon,t}(x):=u_{\varepsilon}\left(\frac{x}{t}\right)$, then we obtain
	\begin{equation}\label{P_27}
		\mathcal{P}(w_{\varepsilon,t})=\frac{N t^N}{2}\|u_{\varepsilon}\|^{2}_2+\frac{N
    t^N}{q}\|u_{\varepsilon}\|^q_{q}-\frac{(N+\alpha)t^{N+\alpha}}{2p}\D_\alpha(|u_\ve|^p, |u_\ve|^p),
	\end{equation}
  and $\mathcal{P}(w_{\varepsilon,1})=\mathcal{P}(u_{\varepsilon}) < 0$.
  On the other hand,  the dependence of $t$ of various terms in Eq. \eqref{P_27} implies 
  that $\mathcal{P}(w_{\varepsilon,t})$ is positive if $t>0$ is small.
	Therefore, by the continuity of $ t\mapsto  \mathcal{P}(w_{\varepsilon,t})$,  there exists $t_{\varepsilon}\in (0,1)$ such that 
$\mathcal{P}(w_{\varepsilon,t_\varepsilon})=0$ and hence 
  $w_{\varepsilon,t_{\varepsilon}}\in \mathscr{P}$. 
  Consequently, 
	\begin{equation}\label{first_ine_29}
		\begin{split}
			\sigma_* \le E(w_{\varepsilon,t_{\varepsilon}}) &  =\frac{\alpha t_{\varepsilon}^{N+\alpha}}{2Np} \D_\alpha(|u_\ve|^p, |u_\ve|^p)
			< \frac{\alpha}{2Np}\D_\alpha(|u_\ve|^p, |u_\ve|^p)+\frac{\varepsilon^\nu}{N}\|\nabla u_{\varepsilon}\|^2_2\\
			& = \mathcal{J}_{\varepsilon}(u_{\varepsilon})=\sigma_{\varepsilon},
		\end{split}
	\end{equation}
  which proves the first part of the statement. 
	
	Now, 
  let $u_*$ be the ground state solution for $(TF)$ obtained in Lemma \ref{rescaling_1}. Then, by the assumption $u_*\in D^{1}(\R^N)$,  
	\begin{equation}
		\mathcal{P}_{\varepsilon}(u_*)=\frac{(N-2){\varepsilon}^{\nu}}{2}\|\nabla u_{*} \|^{2}_2>0.
	\end{equation}
	Define the rescaled function $\omega_{t}(x):=u_*\left(\frac{x}{t}\right)$.
  Then  $ \mathcal{P}_{\varepsilon}(\omega_{t})$, expressed in term of $u_*$ as 
	\begin{equation*}
	\mathcal{P}_{\varepsilon}(\omega_{t})=\tfrac{N-2}{2}\varepsilon^{\nu}t^{N-2}\|\nabla u_{*}\|^{2}_2+ N
    t^{N}\left(\frac{\|u_*\|^{2}_2}{2}+\frac{\|u_*\|^{q}_{q}}{q}\right) -\tfrac{N+\alpha}{2p}t^{N+\alpha}\D_{\alpha}(|u_*|^p, |u_*|^p)
    \to-\infty,
	\end{equation*}
  as $t\to+\infty$.
	 This implies the existence of  $t_{\varepsilon}>1$ such that $\mathcal{P}_{\varepsilon}(\omega_{t_\ve})=0$. In particular, $t_\varepsilon \to 1$ because 
	\begin{equation*}
		1<(t_{\varepsilon})^\alpha\le \frac{N\left(\frac{\|u_*\|^{2}_2}{2}+\frac{\|u_*\|^{q}_{q}}{q}\right)+\frac{N-2}{2}\ve^{\nu}\|\nabla u_*\|^{2}_2}{\frac{N+\alpha}{2p}\D_{\alpha}(|u_*|^p, |u_*|^p)}\to 1\quad \text{as}\  \varepsilon\rightarrow +\infty.
	\end{equation*}
	Now, from  $\ve^{\nu}\to 0$ and $t_\ve\to 1$ as $\ve\to +\infty$, we conclude that
	\begin{equation}\label{eq_20_final}
		\begin{split}
			\sigma_{\varepsilon} & \le \mathcal{J}_{\varepsilon}\left(u_*\big(\tfrac{x}{t_{\varepsilon}}\big)\right)\\
			& =\frac{\varepsilon^{\nu}}{2}t_{\varepsilon}^{N-2}\|\nabla u_*\|^{2}_2
      +\frac{t^{N}_{\varepsilon}}{2}\Big(\|v_*\|^{2}_2+\|u_*\|^{q}_q
      \Big)-\frac{t^{N+\alpha}_{\varepsilon}}{2p}\D_{\alpha}(|u_*|^p, |u_*|^p)\to E(u_*)=\sigma_*.
		\end{split}
	\end{equation}
	The assertion about the convergence $\sigma_\varepsilon \to \sigma_*$ now follows by combining \eqref{first_ine_29} with \eqref{eq_20_final}.
\end{proof}

Next we show that $\sigma_{\ve}\to \sigma_*$ as $\ve\to +\infty$ without assuming that $u_*\in D^{1}(\R^N)$. 
In fact, in the case $p\ge 2$ it is expected that $u_*\not\in D^{1}(\R^N)$, as e.g. follows from H\"older estimates in Lemma \ref{regu_p_big_12}.

\begin{lemma}\label{lemma_6_1}
	Assume $\frac{N+\alpha}{N}<p<\frac{N+\alpha}{N-2}$ and  $q> 2\frac{2p+\alpha}{2+\alpha}$.  
	Then there exists a sequence $(\eps_k)_{k\in\N}$ such that $\eps_k\to\infty$ and  $0<\sigma_{\varepsilon_k}-\sigma_*\rightarrow 0$, as $k\to\infty$.
\end{lemma}
\begin{proof}
	If $p<2$ then the assertion follows by combining Lemma \ref{grad_fast} with Lemma \ref{lemma_6}.
	
	Assume that $p\ge 2$. Again, if the ground state solution $u_*$ of \eqref{eqTF} constructed in Lemma \ref{rescaling_1} belongs to  $D^{1}(\R^N)$ we conclude by Lemma \ref{lemma_6}. If not, (for example for $p>2$), we argue as follows.
	
	First note that arguing as in \eqref{first_ine_29} in the first part of the proof of Lemma \ref{lemma_6}, we conclude that $\sigma_\ve\ge \sigma_*$.
  It remains then to prove that $\sigma_\ve \to \sigma_*$ by constructing an sequence of approximate minimizers of $\mathcal J_{\varepsilon}$ from $u_*$, which is 
  achieved by truncating $u_*$ (to avoid the singularity near the boundary) on a length scale $s$ depending on $\varepsilon$.
	
	Given $s\ge 0$ small, we introduce the cut-off function $\eta_{s}\in C_c^{\infty}(\R^N)$ such that
	$\eta_s(x)=1$ for $|x|\leq R_*-s$, $0<\eta_s(x)<1$ for $R_*-s<|x|\leq R_*-\frac{s}{2}$, $\eta_s(x)=0$ for $|x|\geq R_*-\frac{s}{2}$. Furthermore, $|\eta_s'(x)|\leq \frac{4}{s}$ and $|\eta_s'(x)|\geq \frac{1}{2s}$ for $R_*-\frac{4s}{5}<|x|<R_*-\frac{3s}{5}$.  Set
	$$\psi_s(x):=\eta_s(x)u_*(x).$$
	By the definition of $\eta_s$, since $u_*\in L^{\infty}(\R^N)$ and it is supported in $B_{R_*}$,  for every $1\le r<\infty$ we have 
	\begin{multline}\label{estimate_28}
	\int_{\mathbb{R}^N}|\psi^p_s(x)-u_*^p(x)|^r  =\int_{R_*-s\le |x|\le R_*}|u_*(x)|^{pr}(1-\eta^p_s(x))^{r}\cr
      \le \|u_*\|^{p r}_{L^{\infty}(\R^N)} |A_{R_*-s,R_*}| =\O\left(s\right),
	\end{multline}
	where $|A_{R_*-s,R_*}|$ is the volume of $B_{R_*}\setminus \overline{B}_{R_*-s}$. Further, by combining the Hardy--Littlewood--Sobolev inequality  with \eqref{estimate_28}, we obtain
	\begin{multline*}
	0\le \D_\alpha(|u_*|^p, |u_*|^p) -\D_\alpha(|\psi_s|^p, |\psi_s|^p) 
      =\mathcal{D}_\alpha(|u_*|^p + |\psi_s|^p,|u_*|^p - |\psi_s|^p)
      \cr\le C\|u^p_*+\psi_s^p \|_{\frac{2N}{N+\alpha}} \|u^p_*-\psi^p_s \|_{\frac{2N}{N+\alpha}}=\O\left(s^{\frac{N+\alpha}{2N}}\right).
	\end{multline*}
	To summarise,  the following holds:
	\begin{equation}\label{eq1204}
		\D_\alpha(|\psi_s|^p, |\psi_s|^p)=
		\D_\alpha(|u_*|^p, |u_*|^p)-\O(s^{\frac{N+\alpha}{2N}}),
	\end{equation}
	\begin{equation}\label{eq1205}
		\|\psi_s\|_q^q=\|u_*\|_q^q-\O(s),
	\end{equation}
	\begin{equation}\label{eq1206}
		\|\psi_s\|_2^2=\|u_*\|_2^2-\O(s),
	\end{equation}
	and we recall (see Section \ref{s-notation}) that here $\O(s)$ denotes a {\em nonnegative} function such that $\O(s)\le Cs$  for every $s>0$ small enough and
  for a constant $C>0$ independent of $s$.
	
  Note that by Lemma \ref{smooth_support}, the function $\psi_s$ is smooth and, since $u_*\not\in D^1(\R^N)$, the quantity $\|\nabla \psi_s\|^{2}_2$ blow up as $s\to 0^{+}$. In particular, there exists a decreasing sequence $\left(s_k\right)_{k}$ converging to zero such that  $\|\nabla \psi_{s_k}\|^{2}_2$ diverges monotonically to infinity. Hence,  
  we can define a piecewise linear, monotonically increasing,  continuous function $f:\R_+\to\R_+$ such that $f(0)=0$, $\lim_{s\to 0^{+}}f\left(\frac{1}{s}\right)=+\infty$ and 
		$$f\left(\tfrac{1}{s_k}\right)=\|\nabla \psi_{s_k}\|^{2}_2.$$
		We are going to describe a way we can control the rate of blow up of $f\big(\frac1{s_k}\big)$ in terms of the quantities in \eqref{eq1204}--\eqref{eq1206}.

    In what follows the parameter $s_k$ will be defined as a function of $\ve_k$, so that $f(1/s_k)=\|\nabla\psi_{s_k}\|_2^2$ blows up at a slower rate than $\varepsilon_k^{-\nu}$, 
    to ensure the convergence $\sigma_{\varepsilon_k}\to \sigma_*$.
	To do this, we set 
	\begin{equation}\label{s_epsilon}
		s_k:=\frac{1}{g(\varepsilon_k)},
	\end{equation}
	where $g:\R_+\to\R_+$  is a suitable function such that $\lim_{\varepsilon\to +\infty}g(\varepsilon)=+\infty$, to be chosen later. Then for all sufficiently large $k$ we have
$$\|\nabla \psi_{\frac{1}{g(\varepsilon_k)}}\|_2^2=f(g(\ve_k))\nearrow +\infty\quad \text{as}\ k\to +\infty.$$

   For the sake of notation simplicity we further denote 
    $$\psi_{\varepsilon_k^{-1}}:=\psi_{\frac{1}{g(\varepsilon_k)}}.$$
	Combining together \eqref{eq1204}, \eqref{eq1205} with \eqref{eq1206}, we have that
	\begin{equation}\label{eq_28}
		\begin{split}
			\mathcal{P}_{\varepsilon_k}(\psi_{\eps_k^{-1}})=\left(\frac{N-2}{2}\right)\varepsilon_k^{\nu}\|\nabla \psi_{\varepsilon_k^{-1}}\|^{2}_2+\underbrace{\mathcal{P}(u_*)}_{=0}-\O\left(\tfrac{1}{g(\varepsilon_k)}\right)+\O\left(\left(\tfrac{1}{g(\varepsilon_k)}\right)^{\frac{N+\alpha}{2N}}\right),
		\end{split}
	\end{equation}

	We claim that $\mathcal{P}_{\varepsilon_k}(\psi_{{\varepsilon_k}^{-1}})>0$ for a suitable choice of the function $g$  when $k$ is sufficiently large.  Indeed,
  if $g$ satisfies the condition
	\begin{equation}
		\lim_{\varepsilon\to +\infty}g(\varepsilon) \varepsilon^{\nu}f(g(\varepsilon))=+\infty, 
	\end{equation}
	from \eqref{eq_28} we obtain that
	\begin{equation}\label{positive_29}
		\begin{split}
			\mathcal{P}_{\varepsilon_k}(\psi_{\eps_k^{-1}})& \ge\left(\frac{N-2}{2}\right)\varepsilon_k^{\nu}\|\nabla \psi_{\varepsilon_k^{-1}}\|^{2}_2- \O\left(\tfrac{1}{g(\varepsilon_k)}\right)\\
      & =\left(\frac{N-2}{2}\right)\varepsilon_k^{\nu}f(g(\varepsilon_k))- \O\left(\tfrac{1}{g(\varepsilon_k)}\right)>0,
		\end{split}
	\end{equation}
	provided that $k$ is large enough. Next,  the equality
	\begin{equation*}
		\begin{split}
			& \mathcal{P}_{\varepsilon_k}\left(\psi_{\varepsilon_k^{-1}}\big(\tfrac{x}{t}\big)\right)  =\\
			& =\tfrac{N-2}{2}t^{N-2}\varepsilon_k^{\nu}\|\nabla \psi_{\varepsilon_k^{-1}}\|^{2}_2+N t^{N}\left(\frac{\|\psi_{\varepsilon_k^{-1}}\|^{2}_2}{2}+\frac{\|\psi_{\varepsilon_k^{-1}}\|^{q}_{q}}{q}\right) -\tfrac{N+\alpha}{2p}t^{N+\alpha}\D_\alpha(|\psi_{\ve_k^{-1}}|^p, |\psi_{\ve_k^{-1}}|^p)
		\end{split}
	\end{equation*}
	implies that
	\begin{equation}\label{infinity_29}
		\lim_{t\to +\infty}\mathcal{P}_{\varepsilon_k}\left(\psi_{\varepsilon_k^{-1}}\big(\tfrac{x}{t}\big)\right)=-\infty.
	\end{equation}
	Thus, by combining \eqref{positive_29} with \eqref{infinity_29}, for every $k$  sufficiently large there exists $t_{\varepsilon_k}>1$ such that  
	\begin{equation}\label{eq-t-k}
	\mathcal{P}_{\varepsilon_k}\left(\psi_{\varepsilon_k^{-1}}\big(\tfrac{x}{t_{\varepsilon_k}}\big)\right)=0.
    \end{equation} 
	In particular, by using \eqref{eq1204}, \eqref{eq1205} and \eqref{eq1206}, if $g$ satisfies the second condition
	\begin{equation}\label{cond_29}
		\lim_{\varepsilon\to +\infty}\varepsilon^{\nu}f(g(\varepsilon))=0,
	\end{equation}
  we  then can conclude that $t_{\varepsilon_k}\to 1$, since
	\begin{equation}\label{t_zero}
		\begin{split}
			1<(t_{\varepsilon_k})^{\alpha}\le \frac{N\left(\frac12\|\psi_{\varepsilon_k^{-1}}\|^{2}_2+\frac1{q}\|\psi_{\varepsilon_k^{-1}}\|^{q}_{q}\right)+\frac{N-2}{2}\ve_k^{\nu}\|\nabla \psi_{\varepsilon_k^{-1}}\|^{2}_2}{\frac{N+\alpha}{2p}\D_{\alpha}(|\psi_{\varepsilon_k^{-1}}|^p, |\psi_{\varepsilon_k^{-1}}|^p)}\to 1\quad \text{as}\  k\rightarrow +\infty.
		\end{split}
	\end{equation}
	To summarise, to deduce \eqref{positive_29} and \eqref{t_zero} we need to construct a function $g$ that satisfies: 
	\begin{itemize}
		\item[$(i)$] $\lim_{\varepsilon\to +\infty}\varepsilon^{\nu}f(g(\varepsilon))=0;$\smallskip
		\item[$(ii)$] $\lim_{\varepsilon \to +\infty}g(\varepsilon)\varepsilon^{\nu}f(g(\varepsilon))=+\infty.$
	\end{itemize}
	The existence of such function $g$ is guaranteed by Lemma \ref{tec_lemm} in the Appendix.
	Moreover,  
	\begin{equation}\label{last_ineq_29}
		\begin{split}
			\sigma_{\varepsilon_k} & \le \mathcal{J}_{\varepsilon_k}\left(\psi_{\varepsilon_k^{-1}}\big(\tfrac{x}{t_{\varepsilon_k}}\big)\right)\\
			& =\frac{t_{\varepsilon_k}^{N-2}}{2}\ve_k^{\nu}\| \nabla \psi_{\varepsilon_k^{-1}}\|^{2}_2+t^{N}_{\ve_k}\left(\frac12\|\psi_{\varepsilon_k^{-1}}\|^{2}_2+\frac1{q}\|\psi_{\varepsilon_k^{-1}}\|^{q}_{q}\right)-\frac{t^{N+\alpha}_{\ve_k}}{2p}\D_{\alpha}(|\psi_{\varepsilon_k^{-1}}|^p, |\psi_{\varepsilon_k^{-1}}|^p).
		\end{split}
	\end{equation}
  Finally, in view of \eqref{eq1204}, \eqref{eq1205}, \eqref{eq1206}, \eqref{cond_29} and since $t_{\varepsilon_k}\to 1$, the right hand side of \eqref{last_ineq_29} converges to $\sigma_*$ as $k\to\infty$. 
	This implies that $\sigma_{\varepsilon_k}\rightarrow  \sigma_*$ as $k \rightarrow +\infty.$
\end{proof}

Once the convergence of $\sigma_\varepsilon$ towards $\sigma_*$ is proved, we can show that the term $\varepsilon^\nu \|\nabla u_\varepsilon\|^2_2$
also vanishes in the same limit.

\begin{corollary}\label{gradient_0}
	Assume that $\frac{N+\alpha}{N}<p<\frac{N+\alpha}{N-2}$ and $q>2\frac{2p+\alpha}{2+\alpha}$. 
	Then there exists a sequence $(\eps_k)_{k\in\N}$ and a sequence of ground states $(u_{\varepsilon_k})$ of $(P_{\varepsilon_k})$ such that $\eps_k\to\infty$ and
	$$\varepsilon_k^{\nu} \|\nabla u_{\varepsilon_k}\|^{2}_2\rightarrow 0\quad\text{as $k\to +\infty$}.$$  
\end{corollary}

\begin{proof}
	Arguing as in the first part of Lemma \ref{lemma_6}, there exists $t_\varepsilon\in (0,1)$ such that $u_{\varepsilon}\big(\frac{x}{t_\varepsilon}\big)\in \mathscr{P}$.
	Now, let's consider the sequence $(t_{\ve_k})_k$, corresponding to the sequence $(\ve_k)_k$ in Lemma \ref{lemma_6_1}. We first prove that, up to a subsequence,  $t_{\varepsilon_k}\rightarrow 1$ as $k\to +\infty$. Since $(t_{\varepsilon_k})_k$ is bounded, up to a subsequence $t_{\varepsilon_k}\rightarrow t_0\in [0,1]$. Assume by contradiction that $t_0<1$. Then
	\begin{equation}\label{ineq_24}
		\begin{split}
			\sigma_* &\le E\left(u_{\varepsilon_k}\big(\tfrac{x}{t_{\varepsilon_k}}\big)\right)=\frac{\alpha t_{\varepsilon_k}^{N+\alpha}}{2Np}\D_{\alpha}(|u_{\ve_k}|^p, |u_{\ve_k}|^p)\\
			& \le t^{N+\alpha}_{\varepsilon_k} \mathcal{J}_{\varepsilon_k}(u_{\varepsilon_k})=t^{N+\alpha}_{\ve_k}\sigma_{\ve_k}\rightarrow t^{N+\alpha}_0 \sigma_* <\sigma_*,
		\end{split}
	\end{equation}
	a contradiction.
  Therefore, we have proved that $t_{\ve_k} \to 1$ and furthermore,
	$$\mathcal{J}_{\varepsilon_k}(u_{\varepsilon_k})=\frac{\varepsilon_k^{\nu}}{2}\|\nabla u_{\varepsilon_k}\|^{2}_2+\frac{\alpha}{2Np}\D_{\alpha}(|u_{\ve_k}|^p, |u_{\ve_k}|^p)\rightarrow \sigma_*  .$$
	In particular, 
	$(\varepsilon_k^{\nu}\|\nabla u_{\varepsilon_k}\|^{2}_2)_{k}$, $\left(\D_{\alpha}(|u_{\ve_k}|^p, |u_{\ve_k}|^p)\right)_{k}$ are bounded sequences and the same holds for $\left(\|u_{\varepsilon_k}\|_2\right)_{k}$ and $\left(\|u_{\varepsilon_k}\|_{q}\right)_{k}$.
	Therefore, by combining the equation
	\begin{equation*}
		\begin{split}
			0 & =\tfrac{N}{2}t^N_{\varepsilon_k}\|u_{\varepsilon_k}\|^{2}_2+\tfrac{N}{q}t^N_{\varepsilon_k}\|u_{\varepsilon_k}\|^{q}_q-\tfrac{N+\alpha}{2p}t^{N+\alpha}_{\varepsilon_k} \D_{\alpha}(|u_{\ve_k}|^p, |u_{\ve_k}|^p)\\
			& =-\tfrac{N-2}{2}\varepsilon^{\nu}\|\nabla u_{\varepsilon_k} \|^{2}_2+N(t^{N}_{\varepsilon_k}-1)\left(\frac{\|u_{\varepsilon_k}\|^{2}_2}{2}+\frac{\|u_{\varepsilon_k}\|^{q}_q}{q}\right)-\tfrac{(N+\alpha)}{2p}(t^{N+\alpha}_{\varepsilon_k}-1)\D_{\alpha}(|u_{\ve_k}|^p, |u_{\ve_k}|^p),
		\end{split}
	\end{equation*}
	with boundedness of the above sequences and $t_{\varepsilon_k}\rightarrow 1$, we obtain 
	$\lim_{k\to +\infty}\varepsilon_k^{\nu}\|\nabla u_{\varepsilon_k}\|^{2}_2=0$.
\end{proof}

\begin{proof}[Proof of Theorem \ref{t-TF-0}]
	By Lemma \ref{lemma_6_1} and Corollary \ref{gradient_0}, there exists a sequence $(\ve_k)_{k}$ such that $(u_{\ve_k}(x/t_{\ve_k}))\subset\mathscr P$ is a bounded
  radially nonincreasing minimizing sequence for  the functional $E$. Then, if we set $v_{\ve_k}(x):=u_{\ve_k}(x/t_{\ve_k})$, by arguing as in the proof of Lemma \ref{t-GN-existence}, there exists $\bar{v}\in L^2\cap L^q(\R^N)$ such that 
  \begin{equation*}
  \begin{split}
  &v_{\ve_k}\to \bar{v}\quad \text{in}\ L^s(\R^N),\ \forall s\in (2,q),\\
  &v_{\ve_k}\to \bar{v}\quad \text{a.e.}\  \text{in}\ \R^N.
  \end{split}
  \end{equation*}
We claim that $\bar{v}\in \mathscr{P}$. Indeed, assume by contradiction that $\mathcal{P}(\bar{v})\neq 0$. Since $(v_{\ve_k})_k$ is a minimizing sequence for $\sigma_*$, by the nonlocal Brezis--Lieb Lemma \cite{MMS16}*{Proposition 4.7} we derive 
\begin{equation}\label{nonlocal_24}
	\D_{\alpha}(v^p_{\ve_k}, v^p_{\ve_k})\to  \D_{\alpha}(\bar{v}^p, \bar{v}^p)=\frac{2Np}{\alpha}\sigma_*,
\end{equation}
and by the weak lower semi-continuity of the norms we have $\mathcal{P}(\bar{v})<0$. Furthermore, it is easy to see that there exists $t_0\in (0,1)$ such that $\bar{v}(x/t_0)\in \mathscr{P}$. However this implies that 
\begin{equation*}
\begin{split}
	\sigma_*\le E(\bar{v}(x/t_0))=\frac{\alpha}{2Np}\D_{\alpha}(\bar{v}(x/t_0)^p,\bar{v}(x/t_0)^p)=\frac{t^{N+\alpha}_0 \alpha}{2Np}\D_{\alpha}(\bar{v}^p,\bar{v}^p)<\frac{\alpha}{2Np}\D_{\alpha}(\bar{v}^p,\bar{v}^p)=\sigma_*,
\end{split}
\end{equation*}
that is a contradiction. Hence $\bar{v}\in \mathscr{P}$. 

Consequently, combining the standard Brezis--Lieb lemma with \eqref{nonlocal_24} yields
\begin{equation*}
	\sigma_*=\lim_{k\to +\infty}E(v_{\ve_k})=E(\bar{v})+\lim_{k\to +\infty}\left(\frac{\|v_{\ve_k}-\bar{v}\|^2_2}{2}+\frac{\|v_{\ve_k}-\bar{v}\|^q_q}{q}\right)\ge \sigma_*.
\end{equation*}
This proves that $E(\bar{v})=\sigma_*$ and $(v_{\ve_k})_k$ converges to $\bar{v}$ in $L^2(\R^N)\cap L^q(\R^N)$, i.e. $\bar{v}$ is a ground state solution of \eqref{eqTF}. Finally, from $t_{\ve_k}\to 1$  we further conclude  that $(u_{\ve_k})_k$ converges to $\bar{v}$ as well, and $\bar{v}$ solves \eqref{eqTF}.
\end{proof}


\section{Numerical approximations of the optimisers}
In this section numerical approximations of the optimizers will be briefly
discussed, complementing the theoretical results presented above. In general 
it is difficult to find solutions to nonlinear equations directly, while ground states as the limits of associated evolution equations are much easier to deal with. 
For instance, one might look for ground states to \eqref{eqTF}  by looking 
at solutions of  the parabolic flow
\begin{equation*}
		u_t = (I_\alpha*|u|^p)|u|^{p-2}u - u - |u|^{q-2}u,
\end{equation*}
or gradient (ascent) flow associated with the Rayleigh quotient $\mathcal{R}_\alpha(u)$.
However, the corresponding solutions usually become singular (concentrating at one point)
or zero (and spreading on the whole space). Instead, we consider the alternative
equation 
\begin{equation}	\label{eqTFt}
	u_t = \lambda(t) (I_\alpha*|u|^p)|u|^{p-2}u - u - |u|^{q-2}u,
\end{equation}
where the weight $\lambda(t)$ is chosen to make sure that the total
interaction energy $\mathcal{D}_\alpha(|u|^p,|u|^p)$ is conserved. In other words, 
the condition that $0=\frac{d}{dt}\mathcal{D}_\alpha(|u|^p,|u|^p)$ implies 
\[
\lambda(t) = \left[
\int (I_\alpha * |u|^p)^2|u|^{2p-2}dx
\right]^{-1}\int I_\alpha *|u|^p(|u|^p+|u|^{p+q-2})dx.
\]
Therefore, with fixed (and conserved) $\mathcal{D}_\alpha(|u|^p,|u|^p)$, at $t$ goes to infinity, 
the corresponding solution becomes stationary and $\lambda(t)$ converges to $\lambda(\infty)$.
This stationary solution can then be normalized to a solution of~\eqref{eqTF}.

As usual, the main computational bottleneck in solving Eq.~\eqref{eqTFt} lies in the evaluation of the Riesz potential 
$I_\alpha *\rho$ for some function $\rho$. When $\rho(x)=\rho(|x|)$ is assumed to be radially symmetric, so is $I_\alpha*\rho(x)$,  
and  
\begin{multline*}
	I_\alpha*\rho(|x|) = A_\alpha \int_{\mathbb{R}^N} |x-y|^{\alpha-N}\rho(y)dy \cr
	=A_\alpha (N-1)\omega_{N-1}\int_0^\infty s^{N-1}\rho(s)\int_0^{\pi}
	(r^2+s^2-2rs\cos\theta)^{(\alpha-N)/2}\sin^{N-2}\theta d\theta ds,
\end{multline*}
where $r=|x|$ and $\omega_{N}$ is the surface area of the unit ball in $\mathbb{R}^N$. In terms of the
Gauss hypergeometric function, the previous double integral 
can be simplified, so that   $I_\alpha*\rho(x)$ becomes
\begin{equation}\label{eq:numquad}
\frac{2^{1-\alpha}\Gamma(\frac{N-\alpha}{2})}{\Gamma(\frac{\alpha}{2})\Gamma(\frac{N}{2})}
\int_0^\infty s^{N-1}\rho(s)
(r^2+s^2)^{(\alpha-M)/2}
{}_2F_1
\left(
\frac{N-\alpha}{4},\frac{N-\alpha}{4}+\frac{1}{2};\frac{N}{2};\frac{4r^2s^2}{(r^2+s^2)^2}
\right)ds.  
\end{equation}
At any $r=|x|$, the evaluation of this Riesz potential is essentially reduced to a numerical
quadrature, although the integral in ~\eqref{eq:numquad} becomes singular around $r=s$ for $\alpha \in (0, 1]$ (but still integrable for 
smooth function $\rho$). In three dimension,  
the Riesz potential can be simplified using the fact that
\[
\int_0^{\pi} (r^2+s^2-2rs\cos\theta)^{(\alpha-3)/2}\sin\theta d\theta =
\begin{cases}
	\frac{(r+s)^{\alpha-1}-|r-s|^{\alpha-1}}{rs(\alpha-1)},\qquad & \alpha \neq 1,\cr
	\frac{\ln (r+s)-\ln |r-s|}{rs},\quad &\alpha = 1.
\end{cases}
\]
All the examples shown in this paper are in one or three dimensions only, while the computation
in general dimension using~\eqref{eq:numquad} is less accurate, because of the double integral involved and 
the evaluation of ${}_2F_1$ in the integrand. 

If the optimizer is supported on a finite domain (i.e., $p\geq 2$), 
the numerical quadrature can be performed using standard techniques. When the computational domain is large enough, 
the support of the solution will be selected automatically by the evolution equation~\eqref{eq:numquad}.
When the optimizer is supported on the whole space  (i.e., $p < 2$), 
the solution is expressed on a non-uniform mesh, and the expected
algebraic decay rate  $O(|x|^{-(N-\alpha)/(2-p)})$ 
of the optimizers proved in Corollary~\ref{cor_12} can be used 
as a boundary condition at infinity, to improve the accuracy of the integral in~
\eqref{eq:numquad} by taking into the contribution for $s \in [L,\infty)$ with a computational domain $[0,L]$.

The optimizer for $p=4>2, \alpha = 2.5$ and $q=8$ in three dimension is shown 
in Fig.~\ref{fig:pgt2}, which is discontinuous on the boundary of the support.

\begin{figure}[htp]
\includegraphics[totalheight=0.24\textheight]{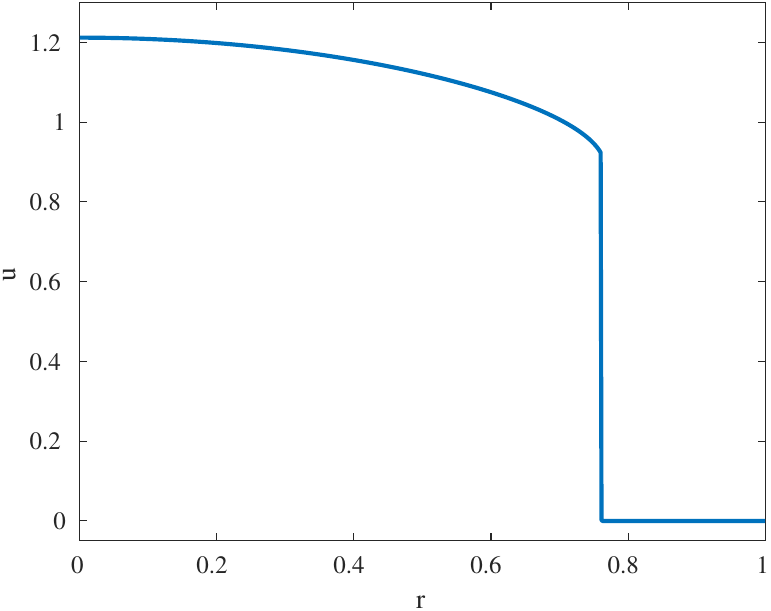}
\caption{The optimiser for $p=4, \alpha = 2.5$ and $q=8$ in three dimension,
which is expected to be compactly supported.}
  \label{fig:pgt2}
\end{figure}

For $p=1.5$, $\alpha = 1$ and $q=2.5$ in three dimension, the optimiser is plotted in Fig~\ref{fig:pls2}. This solution coincides with the explicit family with parameters specified in Eq.~\eqref{eq:explictpq}, with the 
expected algebraic decay rate $O\big(|x|^{-(N-\alpha)/(2-p)}\big)$.
\begin{figure}[htp]
	\begin{center}
		\includegraphics[totalheight=0.24\textheight]{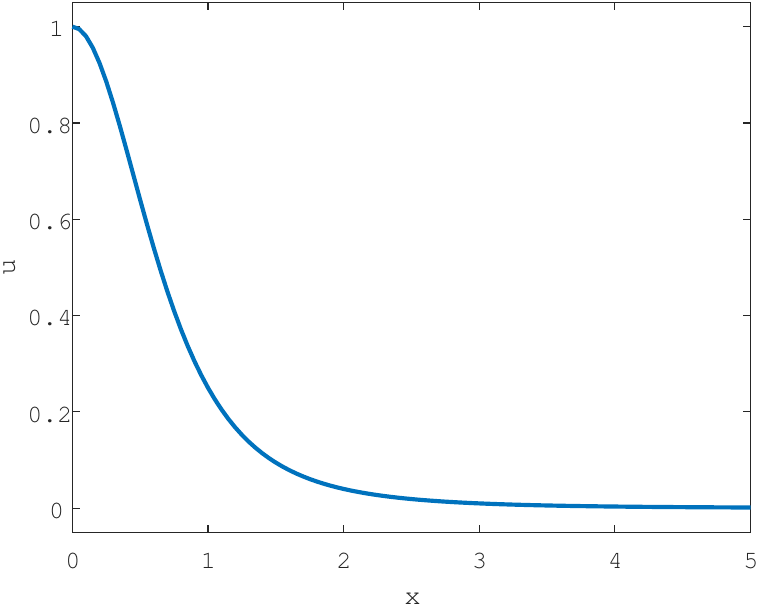}$~~$
		\includegraphics[totalheight=0.24\textheight]{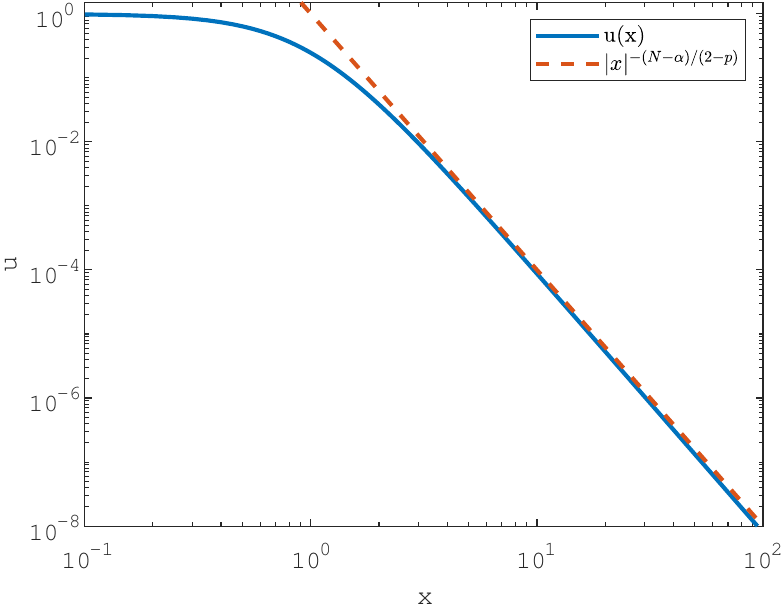}
	\end{center}
\caption{The ground state or optimiser for $p=1.5<2, \alpha=1$ and $q=2.5$ in 
three dimension, bot near the origin or away from the origin. }
  \label{fig:pls2}
\end{figure}

\section{Conclusion and open questions}
In this paper, a special class of Thomas--Fermi type problem was studied, that appears as a limit regime for the Choquard equations with local repulsion. The associated governing equation $(TF)$ was shown to have ground state solutions that correspond to sharp optimizers of a quotient involving the interaction energy and classical Banach norms. Regularity properties and other qualitative information of solutions to this governing equations were investigated. The convergence of ground states of the original Choquard equation in the relevant regimes to a ground state of Thomas--Fermi equation was also proved. 
\smallskip

Below we list several open questions related to the results in the present work.

\subsection{\textmd{\it Uniqueness}}
The uniqueness of the ground state of $(TF)$ beyond the case $p=2$ and $\alpha\le 2$ studied in  \citelist{\cite{Flucher}\cite{Carrillo-Sugiyama}\cite{Volzone}\cite{Carrillo-NA-2021}} seems to be open at present. If the uniqueness in the case $p<2$ is known this would imply that \eqref{eq-Keller-S} provides a sharp constant rather than an upper bound, see Remark \ref{r-sharp}.

\subsection{\textmd{\it Sharp regularity}}
Numerical experiments suggests that for $p>2$ the jump near the boundary $\lambda$ is always strictly larger than $\lambda_*=(\frac{p-2}{q-p})^{1/(q-2)}$, see Remark \ref{r-lambda}. If established analytically, this would also 
rule out the second option in the regularity estimate \eqref{order_hold} implying that H\"older regularity of the ground state near the boundary of the support is always of order $\tau\in (0,\min\left\{\alpha,1\right\})$.

\subsection{\textmd{\it Limit profiles as $\alpha\to 0$ and $\alpha\to N$}}
The fact that the (normalised) sharp constant $\C_{N,\alpha,p,q}$ approaches $1$ in the two limiting cases as $\alpha$ approaches $0$ or $N$, see \eqref{eq-0N}, suggests that after a rescaling ground states of $(TF)$ converge to characteristic functions over a ball. However, numerical experiments indicate that the limits become singular because ground states becomes degenerate or singular. It would be interesting to investigate (singular) limits of the ground states of $(TF)$ as $\alpha\to 0$ and $\alpha\to N$ analytically. 


\appendix

\section{Estimate on $\mathscr{C}_{N,\alpha,p,q}$ for 
  $p=\frac{N+\alpha+2}{N+1}$ and 
  $q = \frac{2(N+2)}{N+1}$
}
\label{app:int}
Here we evaluate the estimate \eqref{C-estimate} on $\mathscr{C}_{N,\alpha,p,q}$ associated with the one-parameter solutions of $(TF)$ discussed in
Remark~\ref{rem:exact}. First using the surface area $2\pi^{N/2}/\Gamma(N/2)$ of the unit sphere in $\R^N$, we have 
for $d>N/2$ the following special integral,
\[
  \int_{\R^N} (1 +|x|^2)^{-d} dx 
  = \frac{2\pi^{N/2}}{\Gamma(N/2)}\int_0^\infty 
r^{N-1}(1+r^2)^{-d}dr
= \frac{\pi^{N/2}\Gamma(d-N/2)}{\Gamma(d)}.
\]
This implies that for the solution $v(x)=(1+|x|^2)^{-(N+2)/2}$ with 
$p=(N+\alpha+2)/(N+1)$ and 
$q = 2(N+2)/(N+1)$,
\[
  \|v\|_2^2 = \frac{\pi^{N/2}\Gamma(N/2+1)}{\Gamma(N+1)},\qquad 
  \|v\|_q^q = \frac{\pi^{N/2}\Gamma(N/2+2)}{\Gamma(N+2)}.
\]
The Riesz potential $I_\alpha *v^p$ can also be evaluated~\cite{Dyda,YanghongPME} as 
\begin{align*}
  I_\alpha *v^p(x) &= 
  2^{-\alpha} \frac{\Gamma(N/2+1)\Gamma((N-\alpha)/2)}{\Gamma((N+\alpha+2)/2)\Gamma(N/2)}
  {}_2F_1\left(\frac{N}{2}+1,\frac{N-\alpha}{2};\frac{N}{2};-|x|^2\right)\cr
  &=
  2^{-1-\alpha}\frac{\Gamma((N-\alpha)/2)}{\Gamma((N+\alpha+2)/2)}
  (1+|x|^2)^{-1-N/2+\alpha/2}(N+\alpha|x|^2),
\end{align*}
and consequently the interaction energy is
\begin{align*}
  \mathcal{D}_\alpha(v^p,v^p)
  &=
  2^{-1-\alpha}\frac{\Gamma((N-\alpha)/2)}{\Gamma((N+\alpha+2)/2)}
  \int_{\mathbb{R}^N} (1+|x|^2)^{-N-2}(N+\alpha|x|^2)dx \cr
  &=\pi^{N/2}\frac{N(N+\alpha+2)}{2^{\alpha+2}} \frac{\Gamma((N-\alpha)/2)\Gamma(N/2+1)}
  {\Gamma((N+\alpha)/2+1)\Gamma(N+2)}.
\end{align*}
Putting all these together, with $\theta = \frac{2\alpha(N+1)}{N(N+\alpha+2)}$, we obtain
\begin{multline*}
  \mathscr{C}_{N,\alpha,p,q}\ge\mathcal{R}_\alpha(v)=\frac{\mathcal{D}_\alpha(v^p,v^p)}{ \|v\|_2^{2p\theta}\|v\|_{q}^{2p(1-\theta)} }\\
  =
  \frac{N(N+\alpha+2)}{\pi^{\alpha/2}2^{\alpha+1}(N+2)} \frac{\Gamma((N-\alpha)/2)}{\Gamma((N+\alpha)/2+1)}
  \left(
    \frac{N+2}{2(N+1)}\frac{\Gamma(N+1)}{\Gamma(N/2+1)}
  \right)^{\alpha/N},
\end{multline*}
which establishes \eqref{C-estimate}.

\section{A Calculus lemma}
Here we prove a technical calculus lemma that was used in the proof of Lemma \ref{lemma_6_1}.

\begin{lemma}\label{tec_lemm}
	For every $\nu<0$, and for every continuous and strictly monotone increasing function $f:\mathbb{R}_{+}\rightarrow \mathbb{R}_{+}$ such that $f(0)=0$, $\lim_{\varepsilon\to +\infty}f(\varepsilon)=+\infty$, there exists a continuous function $g:\mathbb{R}_{+}\rightarrow \mathbb{R}_{+}$
	such that 
	\begin{equation}\label{g_1}
		\lim_{\varepsilon\to +\infty}g(\varepsilon)=+\infty,
	\end{equation}
	\begin{equation}\label{g_2} \lim_{\eps\to +\infty}\varepsilon^{\nu}f(g(\varepsilon))=0,
	\end{equation}
	\begin{equation}\label{g_3}
		\lim_{\varepsilon\to +\infty}g(\varepsilon)\varepsilon^{\nu}f(g(\varepsilon))=+\infty.
	\end{equation}
	
\end{lemma}
\begin{proof}
	Let $H$ be the function defined by 
	\begin{equation}\label{H_small}
		H(\varepsilon)=\min\left\{\log(\varepsilon), \sqrt{f^{-1}\left(\frac{\varepsilon^{-\nu}}{\log(\varepsilon)}\right)}\right\}, \quad \text{for}\ \ve>e^{-\frac{1}{\nu}}, 
	\end{equation}
where $f^{-1}:\R_+\to\R_+$ is the inverse of $f$. Clearly, 	
\begin{equation}\label{H_zero}
	\lim_{\varepsilon\to +\infty}H(\varepsilon)=+\infty.
\end{equation}
Note that such $H$ is continuous and monotone increasing since both  $\sqrt{f^{-1}\big(\frac{\varepsilon^{-\nu}}{\log(\varepsilon)}\big)}$ and $\log(\varepsilon)$ are continuous and monotone increasing functions for $\ve>e^{-\frac{1}{\nu}}$.
	Hence, we define $g$ as follows:
	\begin{equation}\label{g_def} g(\varepsilon):=f^{-1}\left(\frac{\varepsilon^{-\nu}}{H(\varepsilon)}\right)\quad \text{for}\ \ve>e^{-\frac{1}{\nu}},
	\end{equation}
and we extend to a continuous nonnegative function defined on $\R_{+}$.
Note that from \eqref{H_small}, monotonicity and unboundedness of $f^{-1}$,  we have
	\begin{equation*}
		g(\varepsilon)\ge f^{-1}\left(\frac{\varepsilon^{-\nu}}{\log(\varepsilon)}\right)\longrightarrow +\infty\quad \text{as}\ \ve\rightarrow +\infty.
	\end{equation*}
	Hence,  \eqref{g_1} holds.
	Furthermore, from \eqref{H_zero} we obtain
	$$\lim_{\varepsilon\to +\infty}\varepsilon^{\nu}f(g(\varepsilon))=\lim_{\varepsilon\to +\infty}\frac{1}{H(\varepsilon)}=0,$$
	which proves \eqref{g_2}. 
	Finally, again from \eqref{H_small} and \eqref{g_def} it holds
	\begin{equation*}
		\begin{split}
			g(\varepsilon)\varepsilon^{\nu}f(g(\varepsilon))& =\frac{1}{H(\ve)}\cdot f^{-1}\left(\frac{\varepsilon^{-\nu}}{H(\varepsilon)}\right)\ge  \left(f^{-1}\left(\frac{\varepsilon^{-\nu}}{\log(\varepsilon)}\right)\right)^{-\frac{1}{2}} f^{-1}\left(\frac{\varepsilon^{-\nu}}{H(\varepsilon)}\right)\\
			& \ge \sqrt{f^{-1}\left(\frac{\varepsilon^{-\nu}}{\log(\varepsilon)}\right)}\longrightarrow +\infty, \quad \ \text{as}\ \ve\to +\infty,
		\end{split}
	\end{equation*}
	which proves \eqref{g_3}. 
\end{proof}

\vspace{5pt}
{\small
	\noindent{\bf Acknowledgements.}
	D.G.'s research was funded by the EPSRC Maths DTP 2020 Swansea University (EP/V519996/1).
	Y.H. would like to thank V.M. for introducing the problem and the hospitality of Swansea University.
	Z.L. was supported by the National Natural Science Foundation of China (Grant No.12171470).
    This work was initiated during a visit of Z.L.\ at Swansea University and V.M. at Suzhou University of Science and Technology. Hospitality of both institutions is gratefully acknowledged.	
}

\vspace{5pt}

{\small
\noindent{\bf Data availability statement.}
Data sharing not applicable to this article as no datasets were generated or analysed during the current study.}

\bibliographystyle{plain} 
\bibliography{TF}

\begin{bibdiv}
\begin{biblist}

\bib{Auchmuty-Beals}{article}{
      author={Auchmuty, J. F.~G.},
      author={Beals, Richard},
       title={Variational solutions of some nonlinear free boundary problems},
        date={1971},
        ISSN={0003-9527},
     journal={Arch. Rational Mech. Anal.},
      volume={43},
       pages={255\ndash 271},
         url={https://doi.org/10.1007/BF00250465},
}

\bib{BGT-2012}{article}{
      author={Benguria, Rafael~D.},
      author={Gallegos, Pablo},
      author={Tu\v{s}ek, Mat\v{e}j},
       title={A new estimate on the two-dimensional indirect {C}oulomb energy},
        date={2012},
        ISSN={1424-0637,1424-0661},
     journal={Ann. Henri Poincar\'{e}},
      volume={13},
      number={8},
       pages={1733\ndash 1744},
         url={https://doi.org/10.1007/s00023-012-0176-x},
}

\bib{BLS-2008}{article}{
      author={Benguria, Rafael~D},
      author={Loss, Michael},
      author={Siedentop, Heinz},
       title={Stability of atoms and molecules in an ultrarelativistic
  thomas-fermi-weizs{\"a}cker model},
        date={2008},
     journal={Journal of mathematical physics},
      volume={49},
      number={1},
       pages={012302},
}

\bib{BPO-2002}{article}{
      author={Benguria, RD},
      author={P{\'e}rez-Oyarz{\'u}n, S},
       title={The ultrarelativistic thomas--fermi--von weizs{\"a}cker model},
        date={2002},
     journal={Journal of Physics A: Mathematical and General},
      volume={35},
      number={15},
       pages={3409},
}

\bib{Bohmer-Harko}{article}{
      author={Boehmer, Christian},
      author={Harko, Tiberiu},
       title={Can dark matter be a bose--einstein condensate?},
        date={2007},
     journal={Journal of Cosmology and Astroparticle Physics},
      volume={2007},
      number={06},
       pages={025},
}

\bib{Braaten}{article}{
      author={Braaten, Eric},
      author={Zhang, Hong},
       title={Colloquium: The physics of axion stars},
        date={2019},
     journal={Reviews of Modern Physics},
      volume={91},
      number={4},
       pages={041002},
}

\bib{Carrillo-NA}{article}{
      author={Calvez, V.},
      author={Carrillo, J.~A.},
      author={Hoffmann, F.},
       title={Equilibria of homogeneous functionals in the fair-competition
  regime},
        date={2017},
        ISSN={0362-546X,1873-5215},
     journal={Nonlinear Anal.},
      volume={159},
       pages={85\ndash 128},
         url={https://doi.org/10.1016/j.na.2017.03.008},
}

\bib{Carrillo-NA-2021}{article}{
      author={Calvez, Vincent},
      author={Carrillo, Jos\'{e}~Antonio},
      author={Hoffmann, Franca},
       title={Uniqueness of stationary states for singular {K}eller-{S}egel
  type models},
        date={2021},
        ISSN={0362-546X,1873-5215},
     journal={Nonlinear Anal.},
      volume={205},
       pages={Paper No. 112222, 24},
         url={https://doi.org/10.1016/j.na.2020.112222},
}

\bib{Frank}{article}{
      author={Carrillo, Jos\'{e}~A.},
      author={Delgadino, Mat\'{\i}as~G.},
      author={Dolbeault, Jean},
      author={Frank, Rupert~L.},
      author={Hoffmann, Franca},
       title={Reverse {H}ardy-{L}ittlewood-{S}obolev inequalities},
        date={2019},
        ISSN={0021-7824,1776-3371},
     journal={J. Math. Pures Appl. (9)},
      volume={132},
       pages={133\ndash 165},
         url={https://doi.org/10.1016/j.matpur.2019.09.001},
}

\bib{Carrillo-CalcVar}{article}{
      author={Carrillo, Jos\'{e}~A.},
      author={Hoffmann, Franca},
      author={Mainini, Edoardo},
      author={Volzone, Bruno},
       title={Ground states in the diffusion-dominated regime},
        date={2018},
        ISSN={0944-2669,1432-0835},
     journal={Calc. Var. Partial Differential Equations},
      volume={57},
      number={5},
       pages={Paper No. 127, 28},
         url={https://doi.org/10.1007/s00526-018-1402-2},
}

\bib{Carrillo-Sugiyama}{article}{
      author={Carrillo, Jos\'{e}~A.},
      author={Sugiyama, Yoshie},
       title={Compactly supported stationary states of the degenerate
  {K}eller-{S}egel system in the diffusion-dominated regime},
        date={2018},
        ISSN={0022-2518,1943-5258},
     journal={Indiana Univ. Math. J.},
      volume={67},
      number={6},
       pages={2279\ndash 2312},
         url={https://doi.org/10.1512/iumj.2018.67.7524},
}

\bib{Volzone}{article}{
      author={Chan, Hardy},
      author={Gonz\'{a}lez, Mar\'{\i}a del~Mar},
      author={Huang, Yanghong},
      author={Mainini, Edoardo},
      author={Volzone, Bruno},
       title={Uniqueness of entire ground states for the fractional plasma
  problem},
        date={2020},
        ISSN={0944-2669,1432-0835},
     journal={Calc. Var. Partial Differential Equations},
      volume={59},
      number={6},
       pages={Paper No. 195, 42},
         url={https://doi.org/10.1007/s00526-020-01845-y},
}

\bib{Chandrasekhar}{book}{
      author={Chandrasekhar, Subrahmanyan},
      author={Chandrasekhar, Subrahmanyan},
       title={An introduction to the study of stellar structure},
   publisher={Courier Corporation},
        date={1957},
      volume={2},
}

\bib{Chavanis-11}{article}{
      author={Chavanis, Pierre-Henri},
       title={Mass-radius relation of newtonian self-gravitating bose-einstein
  condensates with short-range interactions. i. analytical results},
        date={2011},
     journal={Physical Review D},
      volume={84},
      number={4},
       pages={043531},
}

\bib{Yao-Yao}{article}{
      author={Delgadino, Matias~G.},
      author={Yan, Xukai},
      author={Yao, Yao},
       title={Uniqueness and nonuniqueness of steady states of
  aggregation-diffusion equations},
        date={2022},
        ISSN={0010-3640},
     journal={Comm. Pure Appl. Math.},
      volume={75},
      number={1},
       pages={3\ndash 59},
         url={https://doi.org/10.1002/cpa.21950},
}

\bib{P55}{article}{
      author={du~Plessis, Nicolaas},
       title={Some theorems about the {R}iesz fractional integral},
        date={1955},
        ISSN={0002-9947,1088-6850},
     journal={Trans. Amer. Math. Soc.},
      volume={80},
       pages={124\ndash 134},
         url={https://doi.org/10.2307/1993008},
}

\bib{DuPlessis}{book}{
      author={du~Plessis, Nicolaas},
       title={An introduction to potential theory},
      series={University Mathematical Monographs},
   publisher={Hafner Publishing Co., Darien, CT; Oliver and Boyd, Edinburgh},
        date={1970},
      volume={No. 7},
}

\bib{Dyda}{article}{
      author={Dyda, Bartlomiej},
      author={Kuznetsov, Alexey},
      author={Kwa\'{s}nicki, Mateusz},
       title={Fractional laplace operator and meijer {G}-function},
        date={2017},
        ISSN={0176-4276,1432-0940},
     journal={Constr. Approx.},
      volume={45},
      number={3},
       pages={427\ndash 448},
         url={https://doi.org/10.1007/s00365-016-9336-4},
}

\bib{Eby}{article}{
      author={Eby, Joshua},
      author={Leembruggen, Madelyn},
      author={Street, Lauren},
      author={Suranyi, Peter},
      author={Wijewardhana, LCR},
       title={Approximation methods in the study of boson stars},
        date={2018},
     journal={Physical Review D},
      volume={98},
      number={12},
       pages={123013},
}

\bib{Flucher}{article}{
      author={Flucher, M.},
      author={Wei, J.},
       title={Asymptotic shape and location of small cores in elliptic
  free-boundary problems},
        date={1998},
        ISSN={0025-5874,1432-1823},
     journal={Math. Z.},
      volume={228},
      number={4},
       pages={683\ndash 703},
         url={https://doi.org/10.1007/PL00004636},
}

\bib{PhD}{thesis}{
      author={Greco, Damiano},
       title={Thomas--Fermi type variational problems with low regularity},
        type={PhD Thesis},
 institution={Swansea University, 2024},
}

\bib{TF}{article}{
      author={Greco, Damiano},
       title={Optimal decay and regularity for a Thomas--Fermi type variational
  problem},
        date={2023},
     journal={arXiv preprint arXiv:2302.12586},
}

\bib{YanghongPME}{article}{
      author={Huang, Yanghong},
       title={Explicit {B}arenblatt profiles for fractional porous medium
  equations},
        date={2014},
        ISSN={0024-6093,1469-2120},
     journal={Bull. Lond. Math. Soc.},
      volume={46},
      number={4},
       pages={857\ndash 869},
         url={https://doi.org/10.1112/blms/bdu045},
}

\bib{Lieb}{article}{
      author={Lieb, Elliott~H.},
       title={Sharp constants in the hardy-littlewood-sobolev and related
  inequalities},
        date={1983},
        ISSN={0003-486X},
     journal={Ann. of Math. (2)},
      volume={118},
      number={2},
       pages={349\ndash 374},
}

\bib{Lieb-Loss}{book}{
      author={Lieb, Elliott~H.},
      author={Loss, Michael},
       title={Analysis},
     edition={Second},
      series={Graduate Studies in Mathematics},
   publisher={American Mathematical Society, Providence, RI},
        date={2001},
      volume={14},
        ISBN={0-8218-2783-9},
         url={https://doi.org/10.1090/gsm/014},
}

\bib{Lions-81}{article}{
      author={Lions, P.-L.},
       title={Minimization problems in {$L\sp{1}({\bf R}\sp{3})$}},
        date={1981},
        ISSN={0022-1236},
     journal={J. Functional Analysis},
      volume={41},
      number={2},
       pages={236\ndash 275},
         url={https://doi.org/10.1016/0022-1236(81)90089-6},
}

\bib{Lions-I1}{article}{
      author={Lions, P.-L.},
       title={The concentration-compactness principle in the calculus of
  variations. {T}he locally compact case. {I}},
        date={1984},
        ISSN={0294-1449},
     journal={Ann. Inst. H. Poincar\'{e} Anal. Non Lin\'{e}aire},
      volume={1},
      number={2},
       pages={109\ndash 145},
         url={http://www.numdam.org/item?id=AIHPC_1984__1_2_109_0},
}

\bib{ZLVMeps}{article}{
      author={Liu, Zeng},
      author={Moroz, Vitaly},
       title={Limit profiles for singularly perturbed {C}hoquard equations with
  local repulsion},
        date={2022},
        ISSN={0944-2669,1432-0835},
     journal={Calc. Var. Partial Differential Equations},
      volume={61},
      number={4},
       pages={Paper No. 160, 59},
         url={https://doi.org/10.1007/s00526-022-02255-y},
}

\bib{MMS16}{article}{
      author={Mercuri, Carlo},
      author={Moroz, Vitaly},
      author={Van~Schaftingen, Jean},
       title={Groundstates and radial solutions to nonlinear
  {S}chr\"{o}dinger-{P}oisson-{S}later equations at the critical frequency},
        date={2016},
        ISSN={0944-2669,1432-0835},
     journal={Calc. Var. Partial Differential Equations},
      volume={55},
      number={6},
       pages={Art. 146, 58},
         url={https://doi.org/10.1007/s00526-016-1079-3},
}

\bib{Ros}{article}{
      author={Ros-Oton, X.},
      author={Serra, J.},
       title={Regularity theory for general stable operators},
        date={2016},
        ISSN={0022-0396,1090-2732},
     journal={J. Differential Equations},
      volume={260},
      number={12},
       pages={8675\ndash 8715},
         url={https://doi.org/10.1016/j.jde.2016.02.033},
}

\bib{Silvestre}{article}{
      author={Silvestre, Luis},
       title={Regularity of the obstacle problem for a fractional power of the
  {L}aplace operator},
        date={2007},
        ISSN={0010-3640,1097-0312},
     journal={Comm. Pure Appl. Math.},
      volume={60},
      number={1},
       pages={67\ndash 112},
         url={https://doi.org/10.1002/cpa.20153},
}

\bib{Strauss}{article}{
      author={Strauss, Walter~A.},
       title={Existence of solitary waves in higher dimensions},
        date={1977},
        ISSN={0010-3616,1432-0916},
     journal={Comm. Math. Phys.},
      volume={55},
      number={2},
       pages={149\ndash 162},
         url={http://projecteuclid.org/euclid.cmp/1103900983},
}

\bib{Wang}{article}{
      author={Wang, Xian~Zhi},
       title={Cold bose stars: self-gravitating bose-einstein condensates},
        date={2001},
     journal={Physical Review D},
      volume={64},
      number={12},
       pages={124009},
}

\end{biblist}
\end{bibdiv}

\end{document}